\documentclass[12pt,reqno,amsart]{article}

\usepackage{amsthm,amsfonts,amssymb,amsmath,oldgerm}
\usepackage{graphicx}
\usepackage{caption}
\usepackage{subcaption}
\usepackage{enumerate}
\usepackage{xcolor}
\usepackage{hyperref}
\usepackage[utf8]{inputenc}
\usepackage[T1]{fontenc}     
\usepackage{newtxtext,newtxmath}
\usepackage{xpatch}
\usepackage{ae}
\usepackage{mdframed}
\usepackage{pstricks}
\usepackage[top=25mm, bottom=25mm, left=25mm, right=25mm]{geometry}
\usepackage{fancyhdr}
\usepackage{float}
\usepackage{numprint}
\usepackage{enumitem}
\usepackage[thinlines]{easybmat}

\definecolor{DarkGreen}{rgb}{0,0.3,0}
\definecolor{LightGreen}{rgb}{0.9,1,0.9}
\definecolor{DarkPurple}{rgb}{0.5,0,0.5}
\definecolor{LightPurple}{rgb}{1,0.85,1}
\definecolor{LightRed}{rgb}{1,0.9,0.9}
\definecolor{Green}{rgb}{0,0.7,0}
\definecolor{Mblue}{rgb}{0,0.6,0.6}

\newtheorem{Th}{Theorem}
\newtheorem{coro}{Corollary}

\newtheorem*{remarks}{Remarks}

\newcommand{\Linfty}{L^{\infty}}

\newcommand{\eps}{\varepsilon}

\newcommand{\modif}{\tilde{f}_h} 
\newcommand{\modapp}{f_{app}} 

\newcommand{\RR}{\mathbb{R}}

\newcommand{\hs}{\hspace{2mm}}

\newcommand{\lnorm}{\left|\left|}
\newcommand{\rnorm}{\right|\right|}

\begin{document}

	\title{\rule{160mm}{0.7mm}\\
		\textbf{Machine Learning Methods for Autonomous Ordinary Differential Equations}\\
		\rule{160mm}{0.5mm}}
	
	\author{Maxime Bouchereau$^1$ \and Philippe Chartier$^{3}$ \and Mohammed Lemou$^{2}$ \and Florian Méhats$^4$}
	
	\date{
		\small
		$^1$Université de Rennes\\%
		$^2$Ravel technologies, on leave from INRIA\\%
		$^3$Ravel technologies, on leave from Centre National de la recherche Scientifique (CNRS)\\
            $^4$Ravel technologies, on leave from Université de Rennes\\
		\hs \\
		\large\today	
	}

	\maketitle
	
	\renewcommand{\abstractname}{\Large Abstract}

	\begin{abstract}
		\hs\\
		Ordinary Differential Equations are generally too complex to be solved analytically. Approximations thereof  can be obtained by general purpose numerical methods. However, even though accurate schemes have been developed, they remain computationally expensive: In this paper, we resort to the theory of modified equations in order to obtain "on the fly"  cheap numerical approximations. The recipe consists in approximating, prior to that, the modified field associated to the modified equation by {\em neural networks}.  Elementary convergence results  are then  established and the efficiency of the technique is demonstrated on experiments.
	\end{abstract}

	\hs


	
	
	\section{Introduction}
	
	Ordinary Differential Equations are ubiquitous in the  modelling of systems in domains of science as diverse as biology, dynamics, fluid mechanics, quantum mechanics, thermodynamics or weather forecasting. In most situations, these differential equations can not be solved analytically and necessitate the use of  numerical methods so as to compute accurate approximations \cite{GNI,SODE1}.
	
	Generally speaking, the aforementioned methods are the result of a trade-off between accuracy and computational cost. In simple words, a small approximation error requires long computations. Of course, there are various ways in which one may soften the computational constraints, e.g. by raising the order, taking into account the structure of the problem or optimising the coefficients of the method. Excellent numerical methods are abundant in the literature and we refer to reference books \cite{Sanz-Serna,Butcher,Casas,GNI,SODE1,Reich} for them.
	
	In this paper, we choose to derive as much as possible prior information from the knowledge of the vector field of the equation in order to accelerate the solving of the equation. This is thus a two-step process: (i) first, generate data that will be used during the effective computation of the approximate solution; (ii) second, compute, from an initial value, an approximation of the solution "on the fly" as accurately as possible by using available pre-computed values. As such, the process remains too vague to become practical: this is where the crucial ingredient of the technique comes into play, that is to say modified equations. 
	
	The theory of modified equations has emerged in the context of ordinary differential equations as a powerful tool (referred to as {\em backward error analysis}) to explain the excellent behaviour of structure-preserving numerical methods \cite{GNI}. More recently, modified fields have also been used in a dual manner, as a technique to raise the order of any existing numerical method by twisting appropriately the original vector field \cite{Sanz-Serna}. The idea is that, by adding ad-hoc perturbation terms to the vector field and then solving the associated differential equation, one may compute higher-order approximations. It can be indeed proved that, at least at a formal level, for any numerical method, there exists a modified equation whose solution  by the aforementioned method coincides with the exact solution. The perturbation terms to be added may be obtained in analytical form as {\em elementary differentials} involving various derivatives of the original vector field. Complete expansions are for instance available with the help of representations by trees (known as B-series). However, computing the corresponding expressions analytically would be an extremely tedious process and this is the reason why the technique has remained confined to specific situations \cite{chartier2007modified,GNI,MoserVeselov} of limited practical interest. 
	
	The main idea of this work is thus to combine the theory of modified equations with machine learning techniques and more precisely neural networks. Given a differential equation and a numerical method, the ad-hoc perturbations of the vector field are first {\em learnt} by extensive simulations and then approximated by inference from a neural network. Once this representation of the modified vector field has been obtained, it is used to solve the original equation with the same numerical method for any initial value prescribed in a {\em learnt} domain from the phase-space. The combined computational work ids by far greater than for any usual reasonable numerical method. Nevertheless, if one omits the time spend to learn the perturbation, an accurate solution can be obtained very cheaply as compared to well-established schemes such as those of Dormand \& Prince \cite{DOPRI}.
	
	\subsection{Scope  of the paper}
	
	The article is divided into two main sections: Section \ref{section_theory} is devoted to the exposition of the technique and its convergence analysis, while Section \ref{section_Numerical_experiments} presents numerical experiments illustrating the performances and properties of the schemes we analysed.
	
	Subsection \ref{subsection_General_strategy} exposes the general strategy which is adopted. A specific structure of the neural network is selected, in agreement with the structure of the modified field. Moreover, the details of the machine leaning method for the learning of the modified field are given here, concerning in particular the choice of the training data set and the calibration of the parameters of the neural network (approximating the modified field via {\em loss}-minimization).
	
	Subsection \ref{subsection_General_Theory} establishes a convergence result for the  numerical integration resulting from the combined use of neural networks and classical schemes. The essential difference with standard convergence results is reflected in the multiplicative constant of the local error which turns out to be smaller than for a direct application to the original vector field. A special focus is put on explicit Runge-Kutta methods in the same subsection, where more precise estimates are given.
	
	Moreover, two specific schemes are studied in this subsection: the forward Euler method and a Runge-Kutta method of order $2$, which are two simple occurrences of explicit methods. A short convergence analysis is undertaken for each of them, leading to improved bounds. The same methods are then used for numerical experiments which illustrate two main contributions of this work: (i) The modified field can be learnt efficiently through a neural network, as is illustrated in \ref{subsection_Approximation_modified_field}; (ii) The resulting numerical methods have far greater efficiency, as is illustrated on convergence curves in Subsection \ref{subsection_Loss_Integration}. A comparison with the well-known DOPRI5 method  in Subsection \ref{subsection_Time_Computation} is particularly enlightening with this respect. 

	\subsection{Related work}
	
	The link between differential equations and machine learning has been already explored in several publications. Two approaches are prominent. On the one hand, the ODE vector field can be learnt by the technique of  MSE Loss, which can be applied to both ODEs or PDEs  \cite{raissi2018multistep}. Let us notice also a paper of Burton et al. \cite{brunton2016discovering} which proposes symbolic regression for the  learning of complex dynamical systems. On the other hand, the ODE vector field can be learnt by statistical methods. for instance,  Expectation-Maximization is used in a paper of Nguyen et al. \cite{StatFieldLearning}, while a paper of Raissi et al. \cite{raissi2017machineGaussian} proposes Gaussian processes for linear differential equations.
	
	{\bf Links with modified equations.} Links between machine learning and modified equations have been established more recently, e.g. in the paper \cite{FieldLearningModified} where the theory of  modified equation is used fora rigorous analysis. Offen et al. \cite{offen2022symplectic} consider numerical methods for Hamiltonian ODEs. The methods used therein are statistical methods, namely Gaussian processes.
	
	{\bf Structure of Neural Networks.} In order to preserve geometric properties, specific neural networks, adapted to the structure of the differential equations under consideration have been developed. A first example are Hamiltonian equations for which Hamiltonian neural networks have been developed (see \cite{david2021symplectic} or \cite{jin2020sympnets}   where irregular time observed data can be used). A second example are Poisson systems for which Jin et al. \cite{jin2020learningPoisson} developed Poisson neural networks. Another example of specific ODEs is studied in another recent paper \cite{VPNets}, where VPNets are used to learn the volume-preserving flows.
	
	{\bf ODEs methods for Neural Networks.} In the same way as neural networks are used to solve ODEs, the reciprocal strategy can be pursued: neural networks can indeed be modelled by ODEs and their properties deduced from the corresponding ODE properties. For instance, Lu et al. \cite{lu2018beyond} have developed new neural networks which are discretizations of ODEs by various numerical methods and Haber et al. \cite{haber2017stable} have developed new structures of neural networks depending on the stability properties of the ODEs. Finally, Chen et al. \cite{chen2018neural} have derived new optimization methods of the loss function  from the properties of the corresponding ODEs (the neural network is here again obtained through the discretization of the ODE).
	
	{\bf Approximation by neural networks.} In order to properly approximate functions by neural networks, error estimates have been established. Anastassiou \cite{anastassiou2000quantitative} stated rates of convergence for approximations of functions by networks, according to the number of parameters and the dimension. By considering neural network spaces as functional spaces, Gribonval et al. \cite{gribonval2022approximation} have obtained inclusions of theses spaces in Besov or Lebesgue spaces, according to the number of parameters and a given rate of convergence. In a separate work, Bach \cite{bach2021learning} has given bounds on the approximation error in a Hilbertian setting. Finally, a problem of approximation \cite{mallat2019sciences} which is underlined is the curse of dimensionality, where high-dimensional vector fields are approximated with a slower rate of convergence than low-dimensional vector fields, i.e. for high dimensions, more parameters and more data will be required in order to get a satisfying learning.

	\section{Improving the accuracy of numerical methods with machine learning}\label{section_theory}
	
	Consider an autonomous ordinary differential equation of the form 
	
	$$\left\{ \begin{array}{ccl}
	     \dot{y}(t) & = & f(y(t)) \in \RR^d, \quad t \in [0,T] \\
	     y(0) &=& y_0 
	\end{array} \right. ,$$
	
	\noindent where $f:\RR^d \rightarrow \RR^d$ is assumed to be smooth enough. By Cauchy-Lipschitz theorem, we have existence and uniqueness of a solution for any given initial value $y_0 \in \RR^d$. We wish to approximate the solution over $[0,T]$ at times $t_n = nh$, $0 \leqslant n \leqslant N $, where $h = \frac{T}{N}$ is the time-step and $N$ is the number of discretization points.
	
	\subsection{General strategy}\label{subsection_General_strategy}
	
	As explained in the Introduction section,  we shall approximate the modified vector field with the help of a neural network.
	
	\subsubsection{Modified field}

	 Let us consider $\Phi^f_h(\cdot)$ the numerical flow associated to a given numerical method ($h$ is the time-step of the method) and to the vector field $f$, and assume that it is of order $p$, in the sense that\footnote{Here and in the sequel, $|\cdot|$ denotes a norm on $\RR^d$.}
	 
	 \begin{eqnarray}
		 \underset{0 \leqslant n \leqslant N}{Max}\left| \left(\Phi^f_h\right)^n(y_0) - \varphi^f_{nh}(y_0) \right| & \leqslant & Ch^p \label{Sydney_RKE}
	 \end{eqnarray}
	 for some constants $C>0$. If we modify the field $f$ used in $\Phi_h$, i.e. if we apply the numerical flow $\Phi_h$ with the modified field $\tilde f_h$ instead of $f$, we may obtain a higher-order approximation. In fact, the theory of {\em modified equations} states that it is possible to construct $\tilde f_h$ as a series of powers of $h$ multiplied by appropriately chosen functions (at least as a formal series), in such a way that $(\Phi_h^{\tilde f_h})^n(y_0)$ coincides exactly with $y(t_n)$. The structure of this modified field writes  (see \cite{GNI})
	 
	 \begin{eqnarray}
	 	\tilde{f_h}(y) = f(y) + h^p\sum_{j=1}^{+\infty } h^{j-1}f^{[j]}(y) = f(y) + h^p\sum_{j=1}^{k-1} h^{j-1}f^{[j]}(y)  + h^{k+p-1}R(y,h) \label{HK_RKE} 
	 \end{eqnarray}
	 where  the coefficient-functions $f^{[j]}$ are built upon derivatives of $f$. It can be shown rigorously that the truncation of this formal series $(\ref{HK_RKE})$ obtained by neglecting the $\mathcal{O}\left( h^{k+p-1} \right)$-terms, leads to 
	 
$$
(\Phi^{\tilde{f_h}}_h)^n(y_0) = \varphi^f_{nh}(y_0) + \mathcal{O}\left( h^{k+p-1} \right)
$$
where $k$ denotes the number of terms kept in $\tilde f_h$.
	 \subsubsection{Machine learning methods}
	 
	 The main idea of this paper consists in  approximating the modified field $\tilde{f_h}$ by a neural network approximation $f_{app}(\cdot,h)$ whose structure mimics the structure of the theoretical modified field (\ref{HK_RKE}). More precisely, we shall approximate separately each function $f^{[j]}$ in $(\ref{HK_RKE})$ with a neural network. As could be anticipated, the truncation of (\ref{HK_RKE}) will be echoed by a similar truncation 
	 
	 \begin{eqnarray}
	 	f_{app}(y,h) = f(y) + h^p\sum_{j=1}^{N_t-1} h^{j-1}f_j(y) + h^{N_t+p-1}R_{a}(y,h), \label{Rome_RKE}
	 \end{eqnarray}
	 
	 \noindent where $f_i$, for $1 \leqslant i \leqslant N_{t}-1$,  and $R_{a}$, are Multi Layer Perceptrons. An obvious advantage of this choice is that the corresponding numerical method remains consistent. Let us further notice that learning the perturbation in this way is also well-adapted to the situations where the original vector field possesses a specific structure  \cite{david2021symplectic,jin2020learningPoisson,jin2020sympnets,raissi2018multistep,VPNets}.\\
	 
	 The complete numerical procedure can be decomposed into three main steps : Firstly, data are collected by simulating very accurately the exact flow at various points of the domain. A high number of simulations and a high accuracy are prerequisite for a good approximation of the modified field. Secondly, the different neural networks are trained separately by minimising a prescribed {\em Loss}-function. Eventually, given an initial data $y_0$, an approximation of the exact solution is obtained by applying the same numerical scheme as the one used for the training to the field $f_{app}$. In more details, we follow the three stages:
	 \begin{enumerate}[label=\textbf{\arabic*.}]
	 	
	 	\item \textbf{Construction of the data set:} $K$ initial data $y_0^{(k)}$ are randomly selected into a compact set $\Omega \subset \RR^d$ (where we wish to simulate the solution) with uniform distribution. Then, for all $0 \leqslant k \leqslant K-1$, we compute a very accurate approximation of the the exact flow at times $h^{(k)}$ with initial condition $y_0^{(k)}$, denoted $y_1^{(k)}$. Time steps $h^{(k)}$ are selected in the domain $[h_-,h_+]$ (we actually pick up the value $\log h^{(k)}$ randomly in the domain $[\log h_-,\log h_+]$ with uniform distribution).
	 	
	 	\item \textbf{Training of the neural networks:} We minimize the Mean Squared Error (MSE), denoted $Loss_{Train}$, which measures the difference  between predicted data $\hat{y}_1^{(k,\ell)}$ and ``exact data'' $y_1^{(k,\ell)}$ by computing the optimal NN's parameters over $K_0$ data (where $1 \leqslant K_0 \leqslant K-1$) by resorting to a gradient method:
	 	
	 	\begin{eqnarray}
		 	Loss_{Train} & = & \frac{1}{K_0}\sum_{k=0}^{K_0-1}\frac{1}{h^{(k)^{2p+2}}}\Big| \underbrace{\Phi^{f_{app}(\cdot,h^{(k)})}_{h^{(k)}}\big( y_0^{(k)} \big)}_{=\hat{y_1}^{(k)}} - \underbrace{\varphi^f_ {h^{(k)}}\big( y_0^{(k)} \big)}_{=y_1^{(k)}} \Big|^2 
	 	\end{eqnarray}
	 	
	 	At the same time, we compute the value of another MSE, denoted  $Loss_{Test}$,  which measures the difference  between predicted data $\hat{y}_1^{(k,\ell)}$ and ``exact data'' $y_1^{(k,\ell)}$ for a subset of the initial values which have not been used to train the NNs. The objective of this step is to estimate the performance of the training for "unknown" initial values :
	 	
	 	\begin{eqnarray}
	 		Loss_{Test} & = & \frac{1}{K-K_0}\sum_{k=K_0}^{K-1}\frac{1}{h^{(k)^{2p+2}}}\left| \Phi^{f_{app}(\cdot,h^{(k)})}_{h^{(k)}}\big( y_0^{(k)} \big) - \varphi^f_ {h^{(k)}}\big( y_0^{(k)} \big) \right|^2 
	 	\end{eqnarray}
	 	
	 	If $Loss_{Train}$ and $Loss_{Test}$ exhibit the same decay pattern, one considers that there is no overfitting, that is to say that the neural network model does not fit exactly against its training data and remains able to perform accurately against unseen data, which is its main purpose. 
	 	\item \textbf{Numerical approximation:} At the end of the training, an accurate approximation $f_{app}(\cdot,h)$ of $\tilde{f_h}$ is available. It is then used to compute the successive values of $(\Phi^{f_{app}(\cdot,h)}_{h})^n(y_0)$ for $n=0,\ldots,N$. 
	 \end{enumerate}

    \subsection{Error analysis}\label{subsection_General_Theory}

    In this subsection, we analyse the error resulting from the procedure described in previous Subsection.  More specifically, we state estimates of the global error for any standard numerical method.
\begin{Th} \label{th:gen}
 Let us denote $\Phi^{f_{app}(\cdot,h)}_h$, the flow  of a given numerical scheme $\Phi_h$ of order $p$, applied to the modified field $\modapp(\cdot,h)$ and let us consider the global error
\begin{eqnarray}
e_n & := & \left(\Phi^{f_{app}(\cdot,h)}_h\right)^n ( y_0) - \left(\varphi^f_ h\right)^n ( y_0), \label{Paris_RKE}
\end{eqnarray}
at times $t_n = nh$ for $0 \leqslant n \leqslant N$. Denoting the learning error by
\begin{eqnarray}
\delta & := & \underset{(y,h) \in \Omega \times [h_-,h_+]}{\mathrm{Max}}\frac{\left|\modif(y,h)-\modapp(y,h)\right|}{h^p} \label{NYC_RKE}
\end{eqnarray}
 and assuming that
\begin{enumerate}[label=\textbf{(\roman*)}]
\item For any pair smooth vector fields $f_1$ and $f_2$, we have
\begin{eqnarray} \label{ass1}
\forall 0 \leq h \leq h_+, \quad \lnorm \Phi_h^{f_1} - \Phi_h^{f_2} \rnorm_{\Linfty(\Omega)} & \leqslant & Ch \lnorm f_1 - f_2 \rnorm_{\Linfty(\Omega)} \label{Toronto_General}
\end{eqnarray}
for some positive constant $C$, independent of $f_1$ and $f_2$;
\item For any smooth vector field $f$, there exists a constant $L > 0$ such that
\begin{eqnarray} \label{ass2}
\forall 0 \leq h \leq h_+, \, \forall (y_1,y_2) \in \Omega^2, \quad  \left|\Phi_h^{f}(y_1) - \Phi_h^{f}(y_2) \right| & \leqslant & (1+Lh)\left|y_1-y_2\right|. \label{Beijing_General}
\end{eqnarray}
\end{enumerate}
Then there exist two constants $\tilde C,\tilde {L} > 0$ such that: 
\begin{eqnarray}
\underset{0 \leqslant n \leqslant N}{Max}\left|e_n\right| & \leqslant & \frac{C\delta h^p}{\tilde{L}}\left( e^{\tilde{L} T}-1 \right) \label{Berlin_General}
\end{eqnarray}
\end{Th}
\begin{proof}
The arguments of the proof are completely standard ans thus omitted.
\end{proof}
\begin{remarks}
\begin{enumerate}[label=\textbf{(\roman*)}]
\item The vector fields $f$ and $\modif$ are smooth, respectively  by assumption and by construction. As for $\modapp(\cdot,h)$, it is smooth as well given that  it is obtained through the composition of affine functions $A_1,\cdots,A_{L+1}$ and nonlinear  functions $\Sigma_1,\cdots,\Sigma_L$ (the so-called  activation functions). The output of the NN thus appears to be of the form $A_{L+1} \circ \Sigma_L \circ A_L \circ \cdots \Sigma_1 \circ A_1$ (for $L$ layers). Hence,  if the activation functions are smooth, then so is $\modapp(\cdot,h)$. This is the case for instance if  the $\Sigma_i$'s are the hyperbolic tangent functions.
\item Assumption (\ref{ass1}) is straightforwardly satisfied for all known consistent methods.
\item A similar error estimate holds for a variable step-size implementation of the numerical method $\Phi$:  if we indeed use the sequence of steps $ 0\leq h_{j} \leq h_+$, then $T = h_0 + \cdots, + h_{N-1}$ and $y_{n+1}^* = \Phi_{h_n}^{\modapp(\cdot,h_n)}\left(y_n^*\right)$, then there exist $\tilde C,\tilde{L}>0$ such that:
\begin{eqnarray*}
\underset{0 \leqslant n \leqslant N}{Max}\left|\Phi^{f_{app}(\cdot,h_{n-1})}_{h_{n-1}} \circ \ldots \circ \Phi^{f_{app}(\cdot,h_0)}_{h_0}  ( y_0) - \varphi^f_{h_{n-1}} \circ \ldots \circ \varphi^f_{h_0} ( y_0)\right| & \leqslant & \frac{\tilde C\delta h^p}{\tilde{L}}\left( e^{\tilde{L} T}-1 \right), \label{Berlin_General}
\end{eqnarray*}
where $h = \underset{0 \leqslant j \leqslant N-1}{\mathrm{Max}}h_j$.
\end{enumerate}
\end{remarks}
We now focus on numerical schemes belonging to the class of explicit Runge-Kutta methods, as this allows to specify some of the constants of previous Theorem. Following Remark (i), we shall assume that $f_{app}(\cdot,h)$ is well-defined and smooth on the compact set $\Omega \times [t_-,t_+]$, so that it is Lipschitz with Lipschitz constant  	
\begin{eqnarray}
\lambda & := & \underset{h \in H}{Max}\lnorm d f_{app}(\cdot,h) \rnorm_{\Linfty(\Omega)}. \label{LA_RKE}
\end{eqnarray}
\begin{coro} Suppose that the numerical scheme $\Phi_h$ from Theorem \ref{th:gen} is the Runge-Kutta method with Butcher tableau $(A,b)$ where $A = (a_{i,j})_{1\leqslant j \leqslant i \leqslant s} \in \mathcal{M}_{s}(\RR)$ and $b = (b_j)_{1 \leqslant j \leqslant s} \in \RR^s$. Assume further that $A$ is strictly lower triangular, so that the scheme is explicit, and that it is of order $p$. Then  inequality (\ref{Berlin_General}) of Theorem \ref{th:gen} holds with 
$\tilde C=\alpha$ and $\tilde L=\alpha \lambda$ where		\begin{eqnarray}	
\alpha & = & \lnorm b \rnorm_{1}\left( 1 + \lambda h_+\lnorm A \rnorm_{\infty}e^{\lambda h_+ \lnorm A \rnorm_{\infty}} \right). \nonumber
\end{eqnarray}	
\end{coro}
\begin{remarks}
\begin{enumerate}[label=\textbf{\arabic*.}]
\item As the approximation $f_{app}(\cdot,h)$ of $\tilde{f_h}$ contains an ${\cal O}(h^p)$-error term, the order of the new numerical procedure coincides with the order of the underlying scheme $\Phi_h$. However, as soon as the NN  becomes large, $\delta$ is small enough for the combined procedure to be significantly more accurate then the simple application of $\Phi_h$. 
\item For $N_t=1$, with $1$ hidden layer, we have density of MLP's in $\mathcal{C}^1(\Omega)$ for the Sobolev norm $W^{1,\infty}$ \cite{jin2020sympnets}. If we have $\underset{h \in H}{\mathrm{Max}}\lnorm  \frac{\left|\modif(\cdot)-\modapp(\cdot,h)\right|}{h^p} \rnorm_{W^{1,\infty}(\Omega)} \leqslant \delta$, then we get
\begin{eqnarray}
\lambda & \leqslant & \underset{h \in H}{Max}\lnorm d\tilde{f_h} \rnorm_{\Linfty(\Omega)} + \delta h_+^p.
\end{eqnarray}
\item Error estimates for the Forward Euler and the sol-called RK2 methods may be slightly improved. One has indeed 
\begin{eqnarray*}
\underset{0 \leqslant n \leqslant N}{Max}\left|e_n^*\right|  \leqslant  \frac{\delta h}{\lambda}\left( e^{\lambda T}-1 \right) 
\quad \mbox{ and } \quad 
\underset{0 \leqslant n \leqslant N}{Max}\left|e_n^*\right|  \leqslant  \frac{\delta h^2}{\lambda}\left( e^{\lambda\left(1+\frac{\lambda h_+}{2}\right) T}-1 \right),
\end{eqnarray*}
for respectively Forward Euler  and RK2 methods.
\end{enumerate}
\end{remarks}

        \subsection{An alternative method for parallel training}

        In this subsection, we show how to learn  the modified field in an alternative way. The main idea  consists in training separately each term (say for instance of the modified field for the forward Euler method), by creating different data sets for different time steps $h_1 < \cdots < h_{N_h}$:
\begin{eqnarray}
 y_j & = & y_0 + h_jf(y_0) + h_j^2f^{[1]}(y_0) + \cdots + h_j^{N_t}f^{[N_t-1]}(y_0) + h_j^{N_t+1}R(y_0,h_j).
\end{eqnarray}
Note that in contrast with  previous method, the step-size is not chosen at random for each initial value. We then obtain a linear system which can be solved by using the {\em generalized inverse} of a matrix. The solution of this linear system encompasses the values of $f^{[1]}(y_0)$, ..., $f^{[N_t-1]}(y_0)$ and $R(y_0,h_1)$, ...,   $R(y_0,h_{N_h})$ which correspond to data usable for learning each term of the modified field separately.\\

The main advantages of this method are a shorter training-time  (at least on a parallel machine), thus allowing for a larger number of data,  and a smaller computational time  (again on a parallel machine) when it comes to obtaining the numerical solution from trained values.

	\section{Numerical experiments}\label{section_Numerical_experiments}
	
	In order to illustrate our theoretical results, we have tested the method given in Section \ref{subsection_General_strategy} for two simple dynamical systems used from simple physics:
	
	\begin{enumerate}[label=\textbf{\arabic*.}]
		\item \textbf{The Non-linear Pendulum:} This system describes the movement of a pendulum under the influence of gravity. It is governed by the equations
		
		$$\left\{\begin{array}{c c c}
		    \dot{y_1} & = & -\sin(y_2) \\
		    \dot{y_2} & = & y_1
		\end{array}\right. ,$$
		
		where $y_2$ denotes the angle of the pendulum with respect to the vertical and $y_1$ its angular velocity. Note that the system is Hamiltonian, see  \cite{david2021symplectic,GNI,jin2020sympnets}. Parameters are given in Appendix \ref{Parameters_Pendulum_Euler} for the forward Euler method, Appendix \ref{Parameters_Pendulum_RK2} for the Runge-Kutta 2 method and Appendix \ref{Parameters_Pendulum_midpoint} for the midpoint rule.
		
		\item \textbf{The Rigid Body system:} This is a three-dimensional system which describes the angular rotation of a solid in the physical space
		
		$$\left\{\begin{array}{c c c}
		\dot{y_1} & = & \left(\frac{1}{I_3} - \frac{1}{I_2}\right)y_2y_3 \\
		\dot{y_2} & = & \left(\frac{1}{I_1} - \frac{1}{I_3}\right)y_1y_3 \\
		\dot{y_3} & = & \left(\frac{1}{I_2} - \frac{1}{I_1}\right)y_1y_2
		\end{array}\right. ,$$
		
		where $y_1$, $y_2$ and $y_3$ denote the angular momenta, and $I_1$, $I_2$ and $I_3$ the momenta of inertia \cite{GNI} (we take here $I_1=1$, $I_2=2$, $I_3=3$). It possesses two invariants, the so-called Casimir $C(y) = \frac{1}{2}|y|^2$ and the energy $H(y) = \frac{1}{2}\left( \frac{y_1^2}{I_1}  + \frac{y_2^2}{I_2} + \frac{y_3^2}{I_3}\right)$. Hence, the solution lies at the intersection of the sphere $|y|^2 = |y(0)|^2$ and of the ellipsoïd $\frac{y_1^2}{I_1}  + \frac{y_2^2}{I_2} + \frac{y_3^2}{I_3} = 2 H(y(0))$. The domain $\Omega$ used for training is thus chosen accordingly. Parameters are given in Appendix \ref{Parameters_Rigid_Body_Euler}.
		
	\end{enumerate}
    
    For the {\em Forward Euler} and {\em RK2} methods, the modified field $(\ref{HK_RKE})$ can be computed by recursive formulae based on various derivatives of $f$, see for instance  \cite{chartier2007modified, GNI}. It is represented by a series whose general term $f^{[j]}$ has an explicit -though complicated- expression, which can be compared with its numerical counterpart, obtained by learning it from the data set. For all $y \in \RR^d$, $1 \leqslant j \leqslant k-1$, we have on the one hand
		\begin{eqnarray}
			f^{[1]}(y) & = & \frac{1}{2}df(y)f(y) \label{Rome_Euler} \\
			f^{[j]}(y) & = & \frac{1}{j+1}df^{[j-1]}(y)f(y) \label{Rome_Euler_bis}
		\end{eqnarray}
for the Forward Euler method,	and on the other hand
		\begin{eqnarray}
			f^{[1]}(y) & = & \frac{1}{24}d(df\cdot f)(y)f(y) + \frac{1}{8}df(y)^2f(y) \label{Rome_RK} \\
			f^{[2]}(y) & = & \frac{1}{24}d\left(d\left(df\cdot f\right)\cdot f\right)f(y) - \frac{1}{2}df(y)f^{[1]}(y) - \frac{1}{2}df^{[1]}(y)f(y).  \label{Rome_RK2_bis}
		\end{eqnarray}
for the Runge-Kutta 2 method. Formulas associated to the midpoint method are given in \cite{chartier2007modified}. 
Truncating the formal power series $(\ref{HK_RKE})$ then gives an approximation of the theoretical modified field, which serves as a reference
	\begin{eqnarray}
	 	\tilde{f_h}(y) = f(y) + h^p\sum_{j=1}^{k-1} h^{j-1}f^{[j]}(y)  + \mathcal{O}\left( h^{k+p-1} \right). \label{Approx_Modif_Field}
	 \end{eqnarray}
Here,  $p$ is the order of the numerical method under consideration.
 
	\subsection{Approximation of the modified field}\label{subsection_Approximation_modified_field}
	
	In this subsection, we study the approximation error between the learned modified field $(\ref{Rome_RKE})$ and the theoretical modified field $(\ref{HK_RKE})$ for the nonlinear Pendulum. We observe the learning error w.r.t. both space and time step variables. More precisely, we plot 
	the function
	\begin{eqnarray}
		g_h^k :x & \mapsto & \frac{1}{h^p}\left| \tilde{f_h}^k(x) - f_{app}(x,h) \right|
	\end{eqnarray}
	for $k=4$ over the domain $\Omega = [-2,2]^2$ in order to study the learning error in space, where the ${\cal O}$-term in $\tilde{f_h}(y)$ is simply neglected. We furthermore represent $\underset{\Omega}{\mathrm{Max}}g_h^k$ for several values of time steps $h$, in order to study the learning error in function of the the time step. Note that we clearly get  the expected order of convergence of $f_{app}(\cdot,h)$ towards the modified field $\tilde{f_h}$, with the exception of a plateau for small values of $h$ .\\
	
	Figures \ref{fig_learning_error_pendulum_Euler},\ref{fig_learning_error_pendulum_RK2} and \ref{fig_learning_error_pendulum_midpoint} show that the error $g_h^4$ is globally constant at the center of the domain and grows near its boundaries (see  \cite{david2021symplectic} where a similar behaviour is observed). 
	
	Altogether, these experiments confirm that the modified field can be appropriately learned with our neural network.

	\begin{figure}[H]
		
		\begin{minipage}{1.1\linewidth}
			\centering
			\begin{minipage}{0.45\linewidth}
				\begin{figure}[H]
					\includegraphics[width=\linewidth]{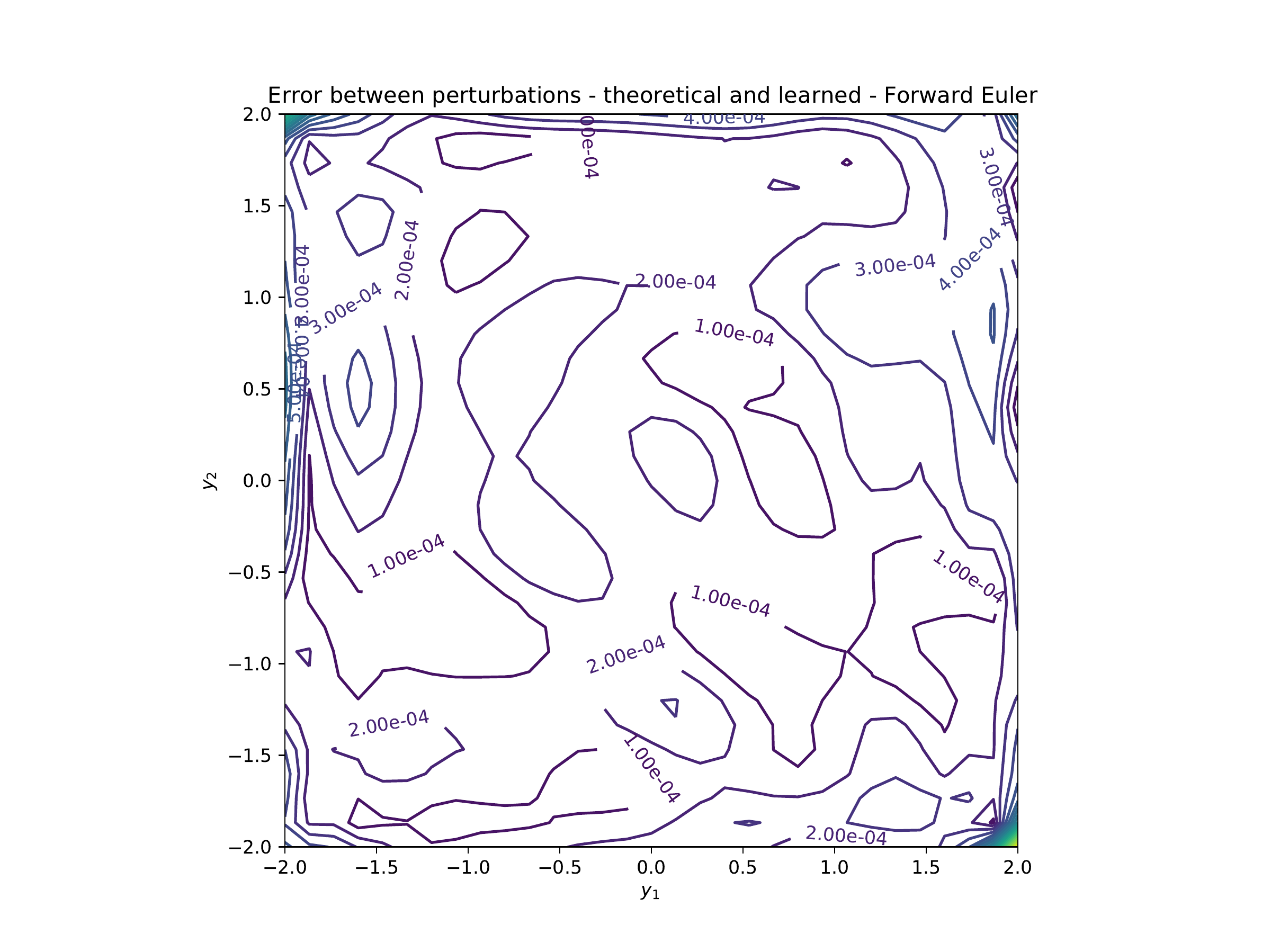}
				\end{figure}
			\end{minipage}
			\begin{minipage}{0.45\linewidth}
				\begin{figure}[H]
					\includegraphics[width=\linewidth]{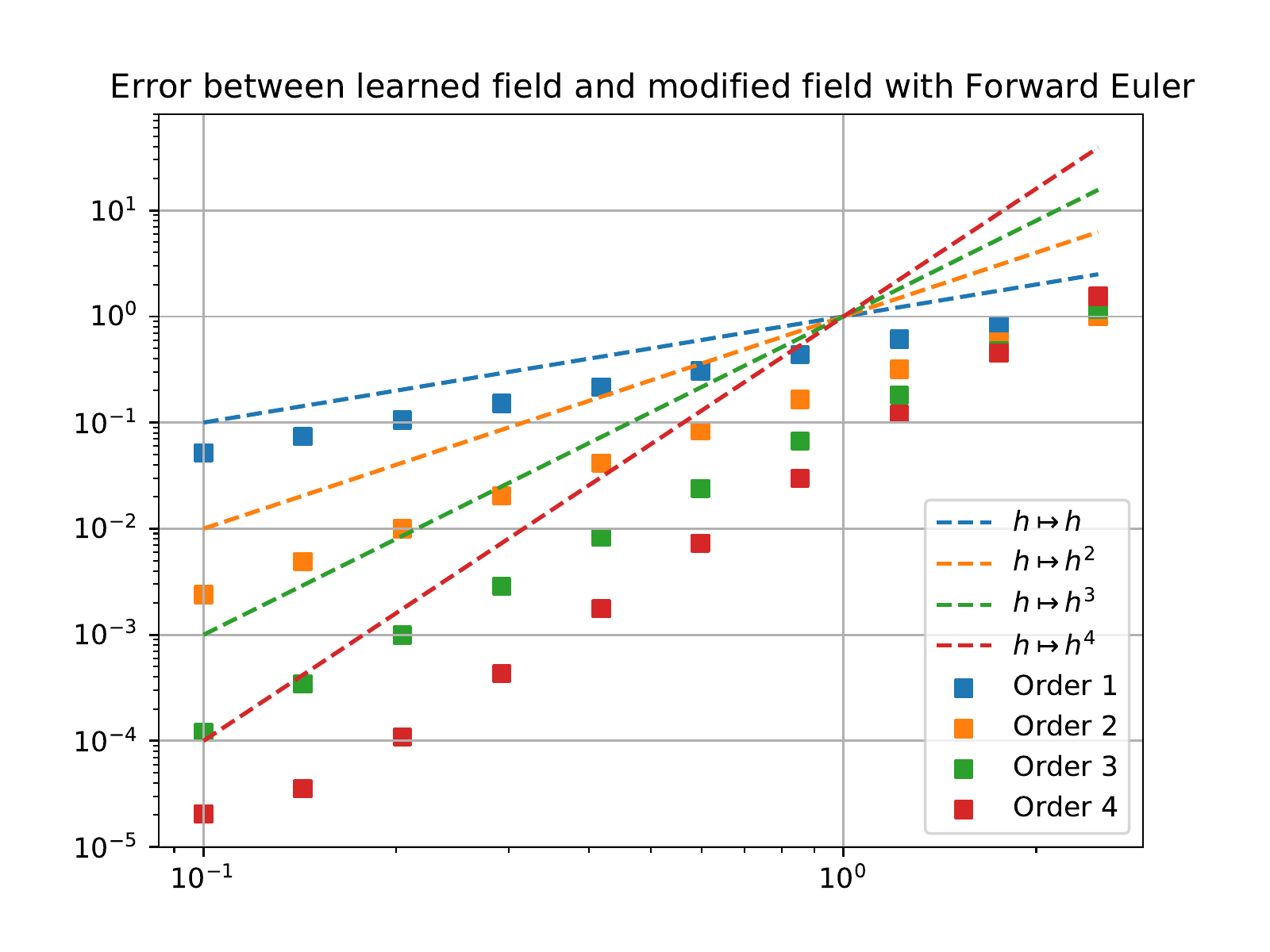}
				\end{figure}
			\end{minipage}
			
		\end{minipage}
		\caption{Forward Euler method. Left: Difference between $\tilde{f_h}^4$ and $f_{app}(\cdot,h)$ for $h=0.1$. Right: Error between $\tilde{f_h}^k$ and $f_{app}(\cdot,h)$ for $1 \leqslant k \leqslant 4$. }
		\label{fig_learning_error_pendulum_Euler}
	\end{figure}
	
	\begin{figure}[H]
		
		\begin{minipage}{1.1\linewidth}
			\centering
			\begin{minipage}{0.45\linewidth}
				\begin{figure}[H]
					\includegraphics[width=\linewidth]{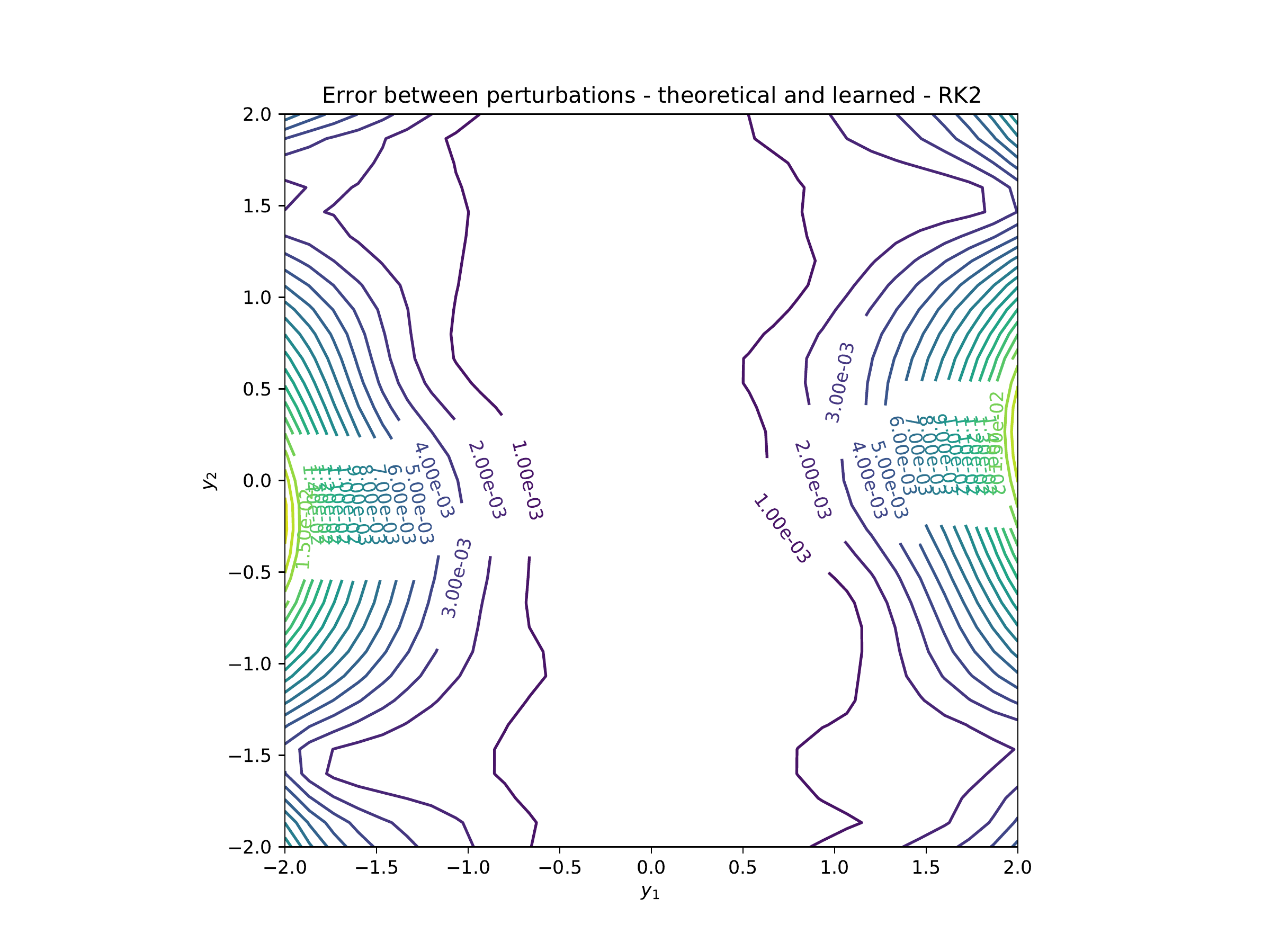}
				\end{figure}
			\end{minipage}
			\begin{minipage}{0.45\linewidth}
				\begin{figure}[H]
					\includegraphics[width=\linewidth]{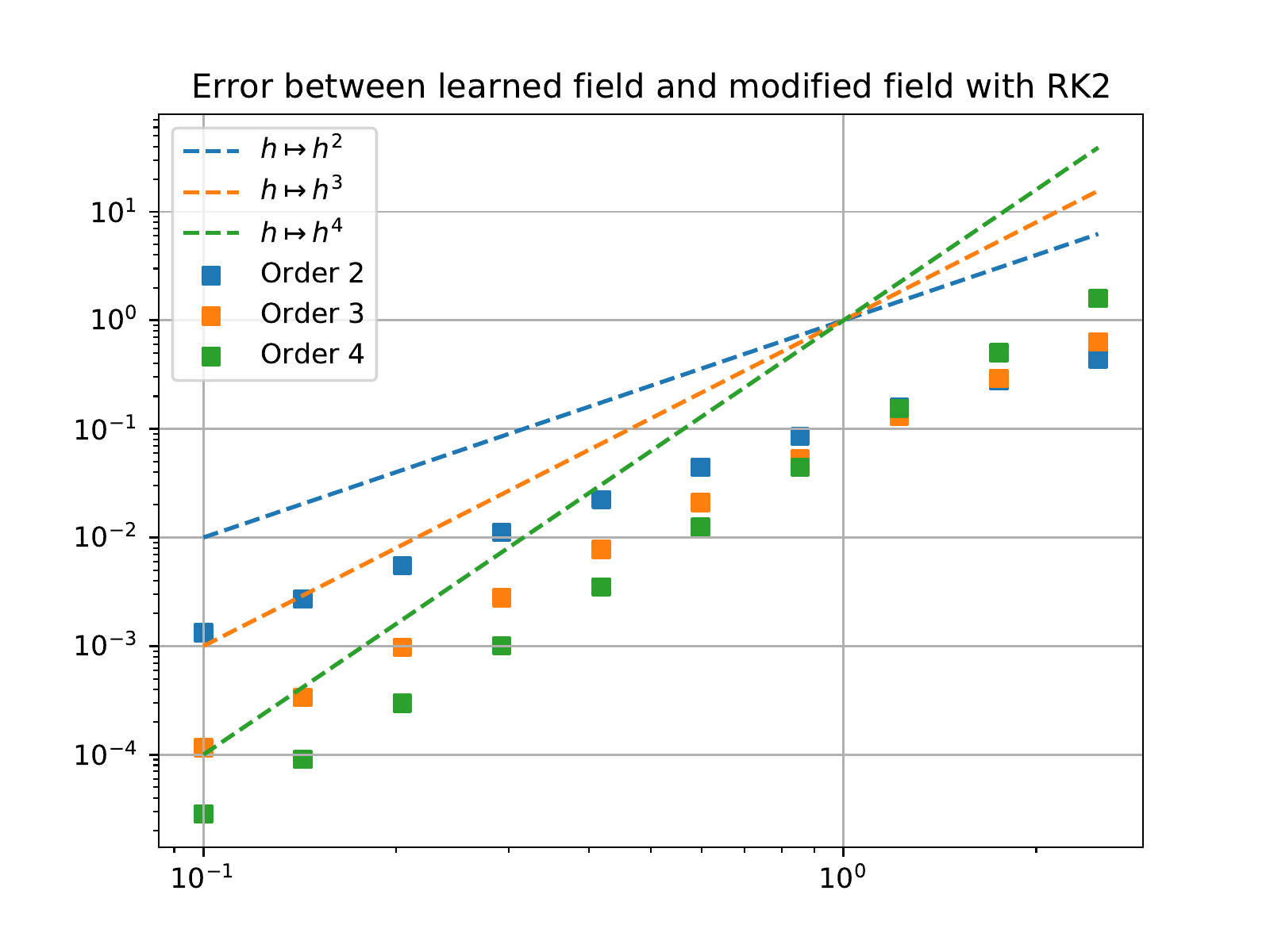}
				\end{figure}
			\end{minipage}
			
		\end{minipage}
		\caption{Runge-Kutta 2 method. Left: Difference between $\tilde{f_h}^4$  and $f_{app}(\cdot,h)$ for $h=0.1$. Right: Error between $\tilde{f_h}^k$ and $f_{app}(\cdot,h)$ for $1 \leqslant k \leqslant 4$. }
		\label{fig_learning_error_pendulum_RK2}
	\end{figure}
	
	\begin{figure}[H]
		
		\begin{minipage}{1.1\linewidth}
			\centering
			\begin{minipage}{0.45\linewidth}
				\begin{figure}[H]
					\includegraphics[width=\linewidth]{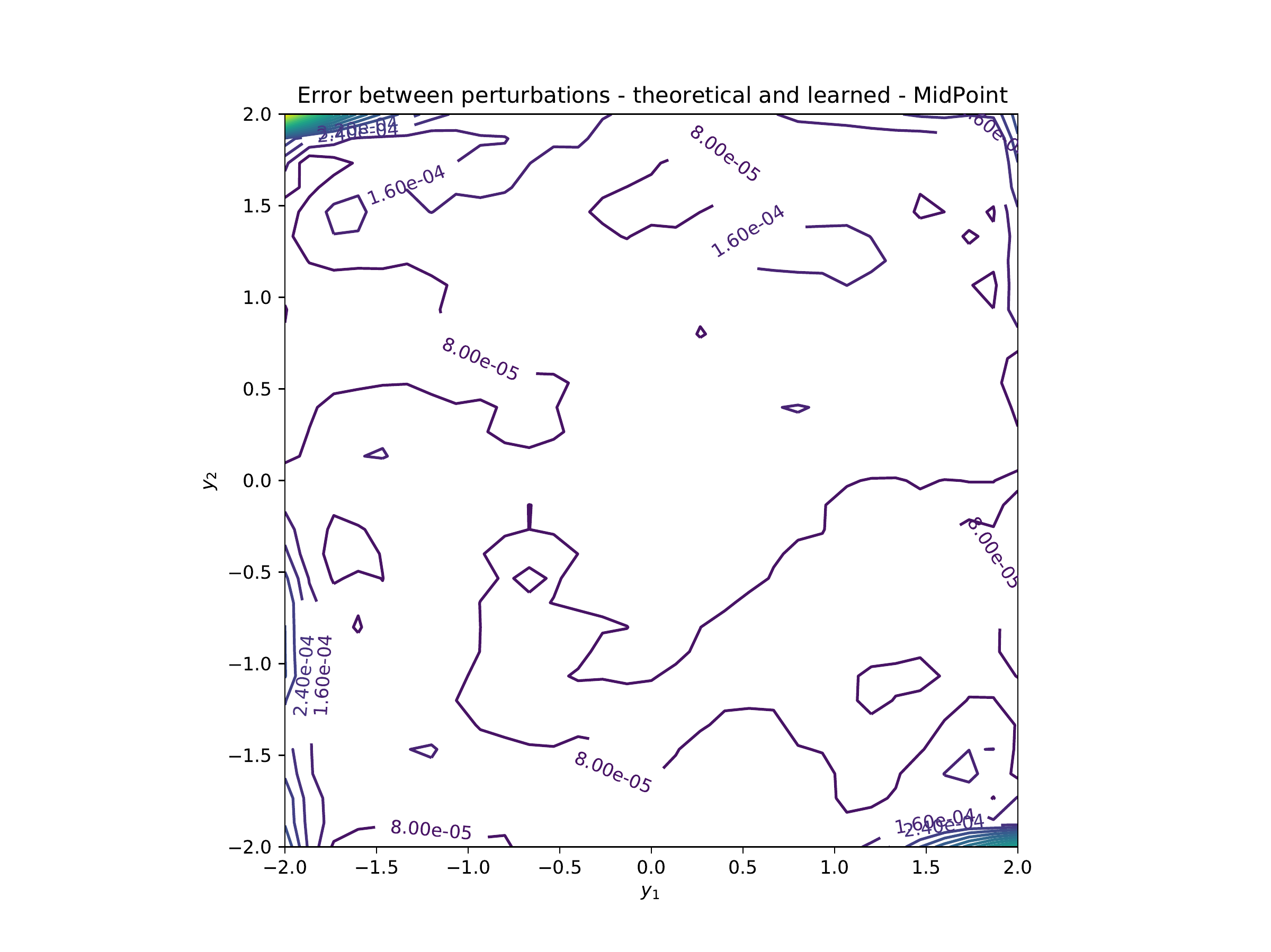}
				\end{figure}
			\end{minipage}
			\begin{minipage}{0.45\linewidth}
				\begin{figure}[H]
					\includegraphics[width=\linewidth]{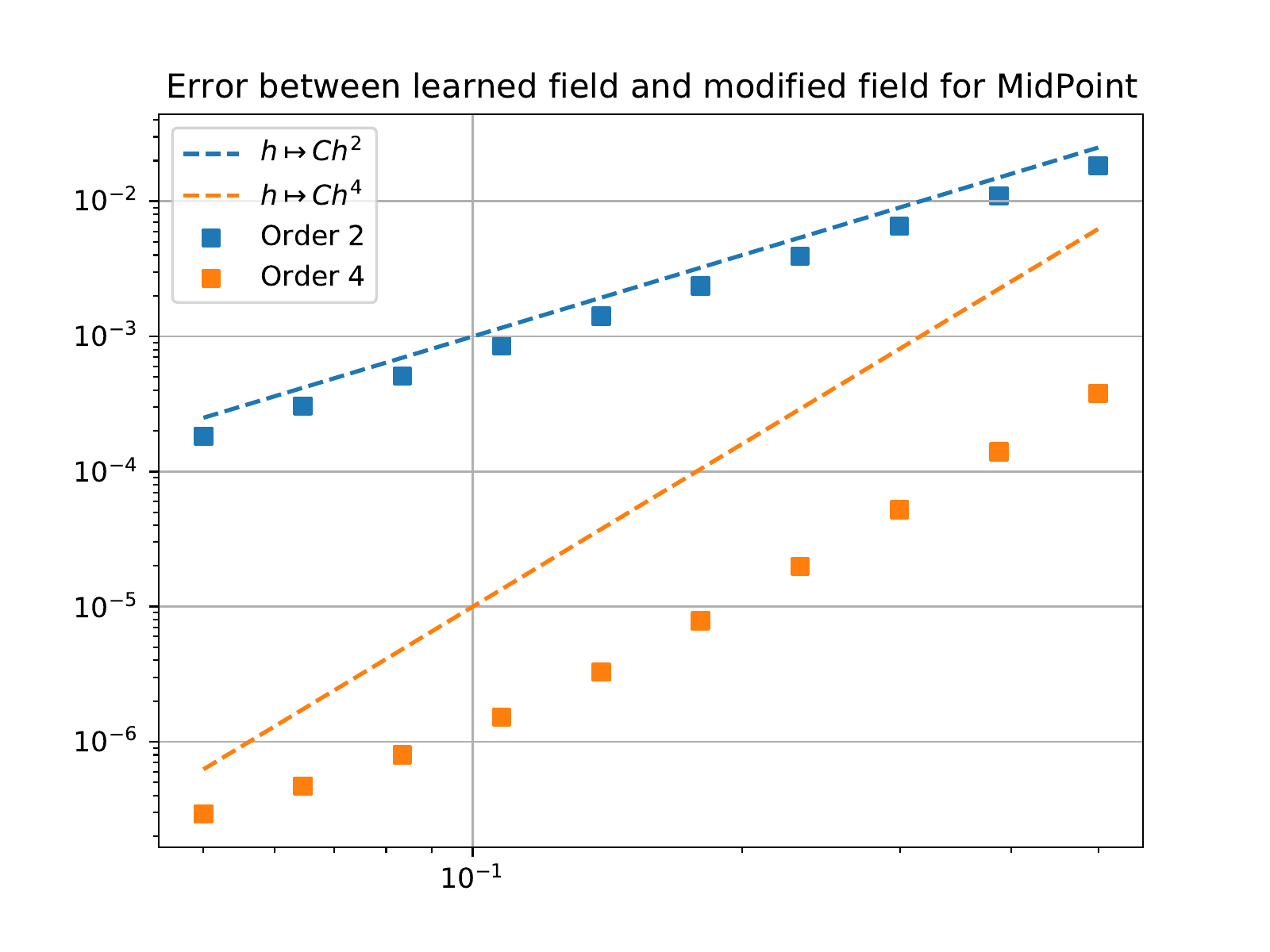}
				\end{figure}
			\end{minipage}
			
		\end{minipage}
		\caption{Midpoint method. Left: Difference between $\tilde{f_h}^4$ and $f_{app}(\cdot,h)$ for $h=0.05$. Right: Error between $\tilde{f_h}^k$ and $f_{app}(\cdot,h)$ for $1 \leqslant k \leqslant 4$. Note that owing to the structure of the modified field for midpoint method (see \cite{chartier2007modified,GNI}), terms for odd powers of $h$ vanish.}
		\label{fig_learning_error_pendulum_midpoint}
	\end{figure}

	\subsection{Loss decay and Integration of ODE's}\label{subsection_Loss_Integration}
	
	Now, in order to compare, for a given method,  the integration of a dynamical system  with the original field and with the learned modified field, we will solve the nonlinear Pendulum with the Forward Euler, Runge-Kutta 2 and midpoint methods and the Rigid Body system with the Forward Euler method.
	However, prior to that, we study the decays of the $Loss$-functions  for the training and testing data sets  ($Loss_{Train}$ and $Loss_{Test}$). Their similarity is a good indication that there is no overfitting (the size of the training data set is thus appropriately estimated). As the MSE $Loss$ is used, it gives an idea of the value of the square of the learning error.\\

    Figures \ref{fig_trajectories_pendulum_Euler}, \ref{fig_trajectories_pendulum_RK2} and \ref{fig_trajectories_rigid_body} show a more accurate numerical integration by using the corresponding learned modified field $f_{app}(\cdot,h)$ than using $f$. Moreover, exact flow and numerical flow with $f_{app}(\cdot,h)$ seem identical due to the small numerical error.

	 \begin{figure}[H]
            \centering
            \includegraphics[scale = 0.6]{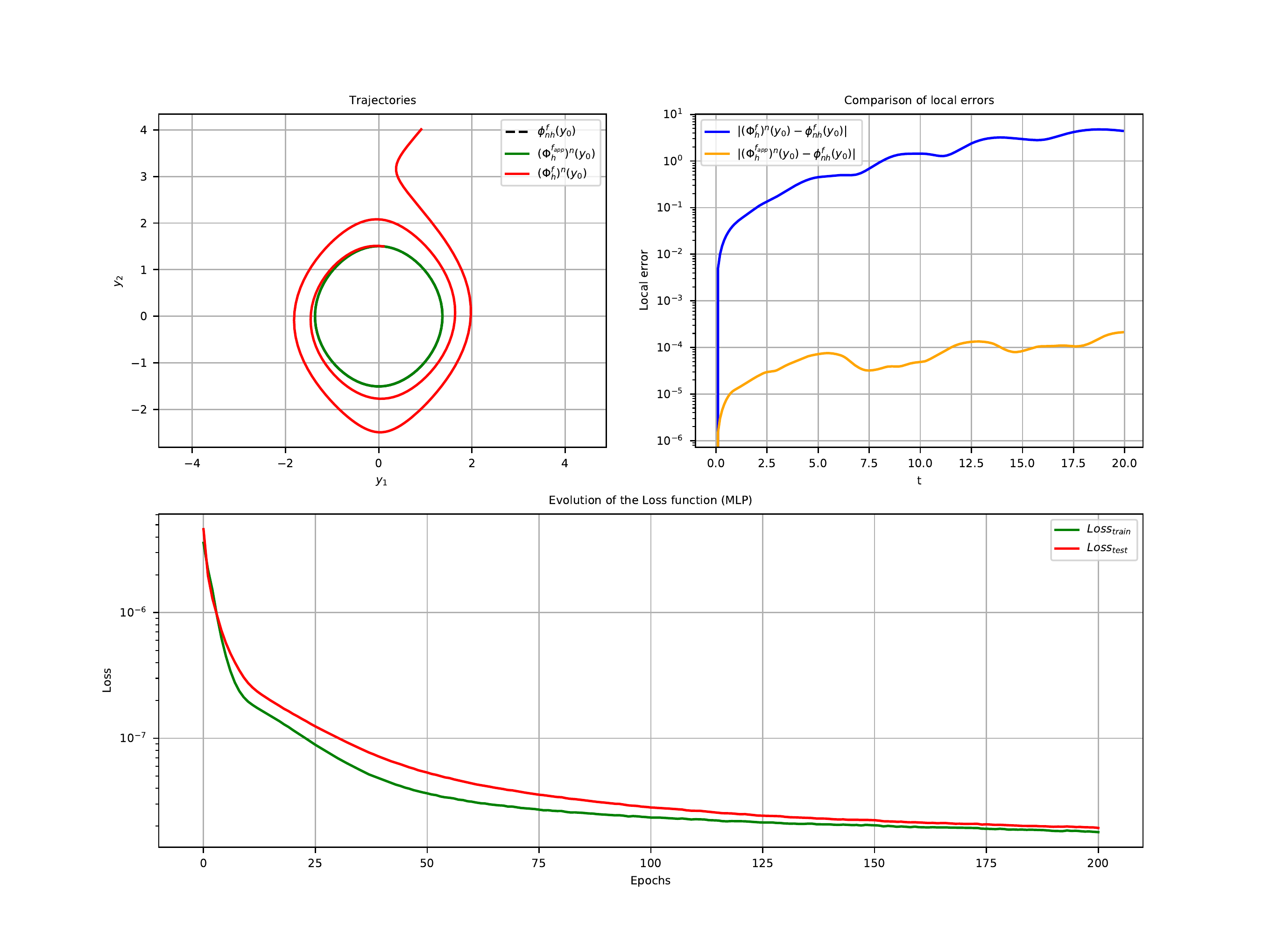}
            \caption{Comparison between $Loss$ decays (green: $Loss_{Train}$, red: $Loss_{Test}$), trajectories (dashed dark: exact flow, red: numerical flow with $f$, green: numerical flow with $f_{app}(\cdot,h)$) and local error (blue: exact flow and numerical flow with $f$, yellow: exact and numerical flow with $f_{app}(\cdot,h)$ ) for the nonlinear pendulum with Forward Euler method.}
            \label{fig_trajectories_pendulum_Euler}
        \end{figure}

        \begin{figure}[H]
            \centering
            \includegraphics[scale = 0.6]{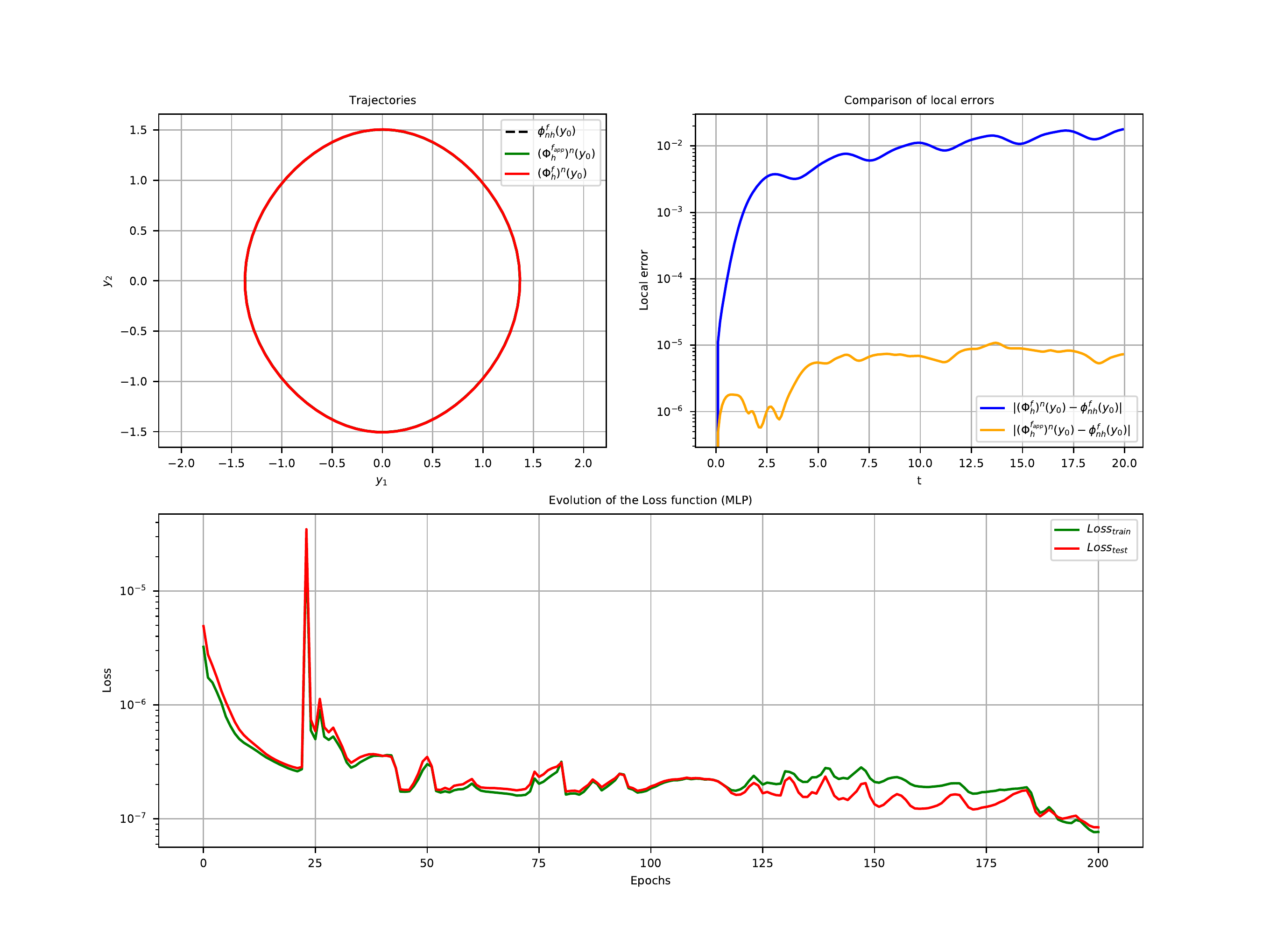}
            \caption{Comparison between $Loss$ decays (green: $Loss_{Train}$, red: $Loss_{Test}$), trajectories (dashed dark: exact flow, red: numerical flow with $f$, green: numerical flow with $f_{app}(\cdot,h)$) and local error (blue: exact flow and numerical flow with $f$, yellow: exact and numerical flow with $f_{app}(\cdot,h)$ ) for the nonlinear pendulum with Runge-Kutta 2 method.}
            \label{fig_trajectories_pendulum_RK2}
        \end{figure}

        \begin{figure}[H]
            \centering
            \includegraphics[scale = 0.6]{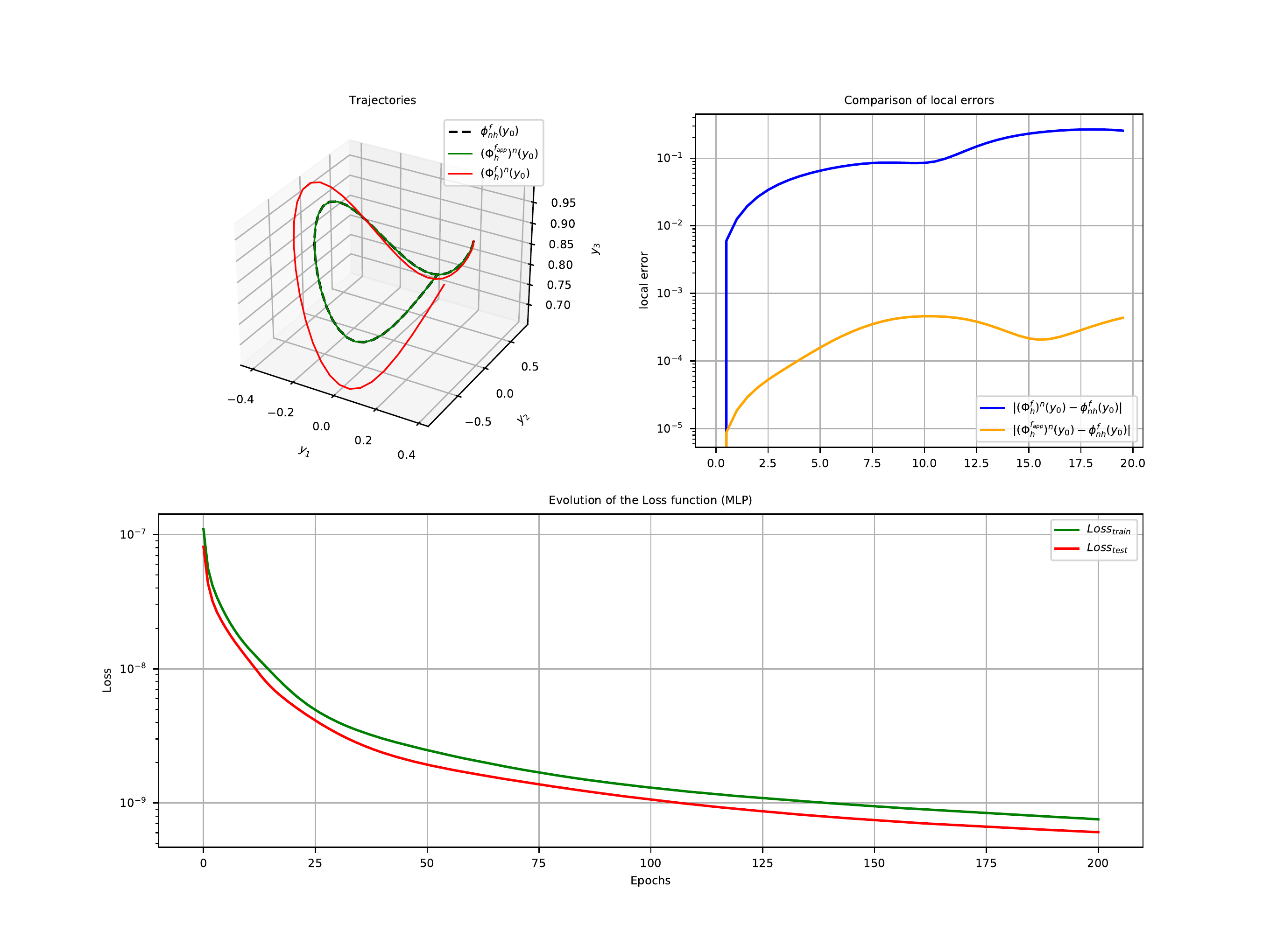}
            \caption{Comparison between $Loss$ decays (green: $Loss_{Train}$, red: $Loss_{Test}$), trajectories (dashed dark: exact flow, red: numerical flow with $f$, green: numerical flow with $f_{app}(\cdot,h)$) and local error (blue: exact flow and numerical flow with $f$, yellow: exact and numerical flow with $f_{app}(\cdot,h)$ ) for the Rigid Body system with Forward Euler method.}
            \label{fig_trajectories_rigid_body}
        \end{figure}

        As the midpoint method is a symmetric and symplectic method, it is known to preserve accurately the geometric properties of the model. In order to evaluate the extent to which this feature persists in our context, we simply plot the value of the Hamiltonian along the numerical solution obtained from learned data. We test this method for the pendulum system, which is hamiltonian. Figure \ref{fig_trajectories_pendulum_MP} shows a smaller error for integration with $f_{app}(\cdot,h)$ by using the midpoint method than integration with $f$. Moreover, the hamiltonian function of the pendulum system, given by\\
        
        \begin{eqnarray}
            H : y & \mapsto & \frac{1}{2}y_1^2 + \left(1 - \cos(y_2)\right)
        \end{eqnarray}
        
        is preserved by the midpoint method with $f_{app}(\cdot,h)$ with smaller oscillations than midpoint with $f$. Preservation is better than DOPRI5 too, which is a non-symplectic method, as shown in Figure \ref{fig_hamiltonian_evolution_pendulum_MP}.
	
	\begin{figure}[H]
            \centering
            \includegraphics[scale = 0.6]{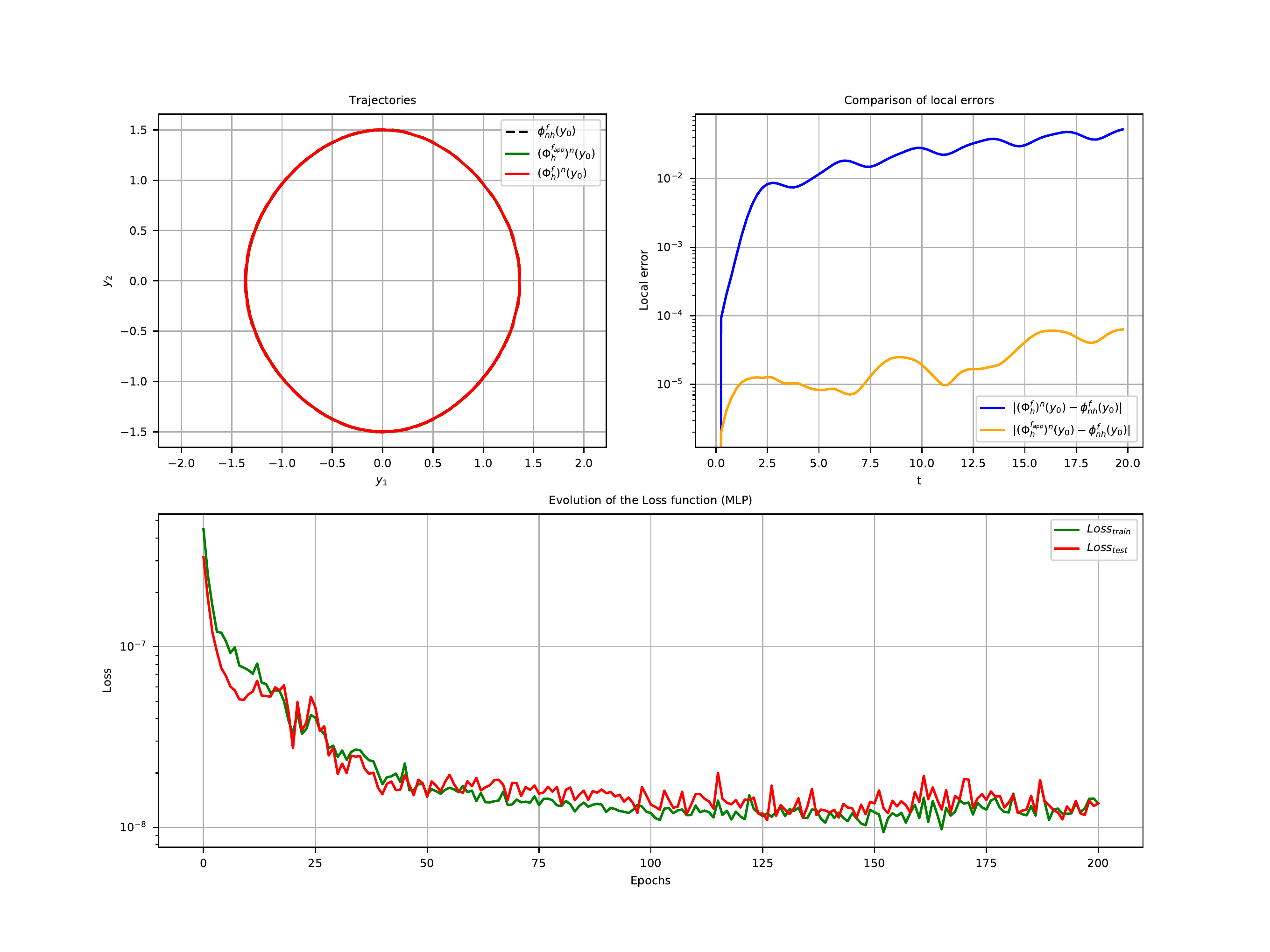}
            \caption{Comparison between $Loss$ decays (green: $Loss_{Train}$, red: $Loss_{Test}$), trajectories (dashed dark: exact flow, red: numerical flow with $f$, green: numerical flow with $f_{app}(\cdot,h)$) and local error (blue: exact flow and numerical flow with $f$, yellow: exact and numerical flow with $f_{app}(\cdot,h)$ ) for the nonlinear pendulum with midpoint method.}
            \label{fig_trajectories_pendulum_MP}
        \end{figure}

        \begin{figure}[H]
            \centering
            \includegraphics[scale = 0.8]{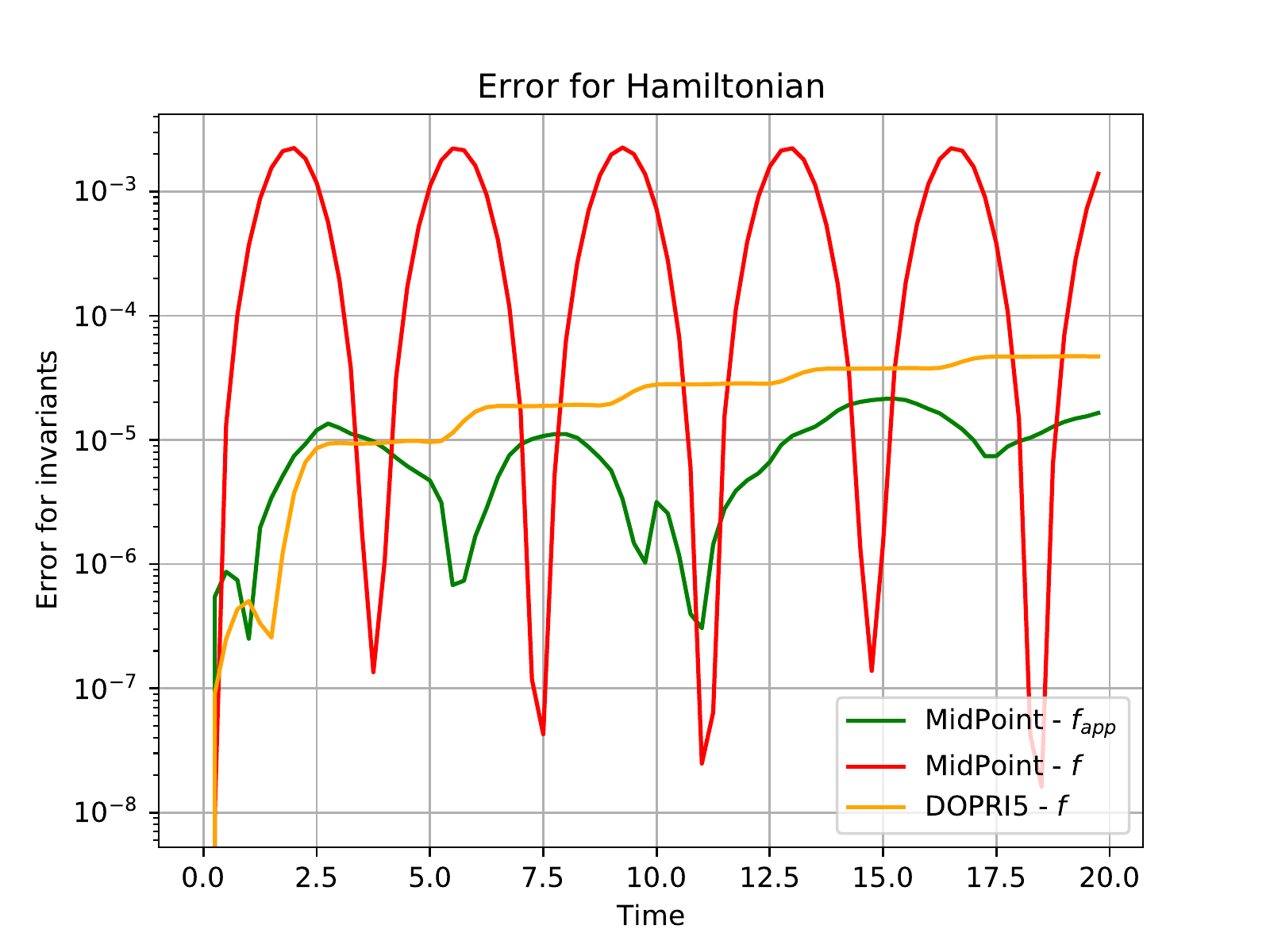}
            \caption{Evolution of the error between Hamiltonian $H:y \mapsto (1-\cos(y_2)) + \frac{1}{2}y_1^2$ over the numerical flow and Hamiltonian at $t=0$, $H(y_0)$.}
            \label{fig_hamiltonian_evolution_pendulum_MP}
        \end{figure}

	Eventually, we study the global error between the exact flow and the approximation obtained from the original field, as well as the error between the exact flow and the numerical flow obtained form the learned modified field. The errors are plot as functions of the step-size and the curves are in perfect agreement with the estimates of previous theorems (for the Forward Euler, Runge-Kutta $2$ and midpoint methods).
	
	\begin{figure}[H]
		
		\begin{minipage}{1.1\linewidth}
			\centering
			\begin{minipage}{0.45\linewidth}
				\begin{figure}[H]
					\includegraphics[width=\linewidth]{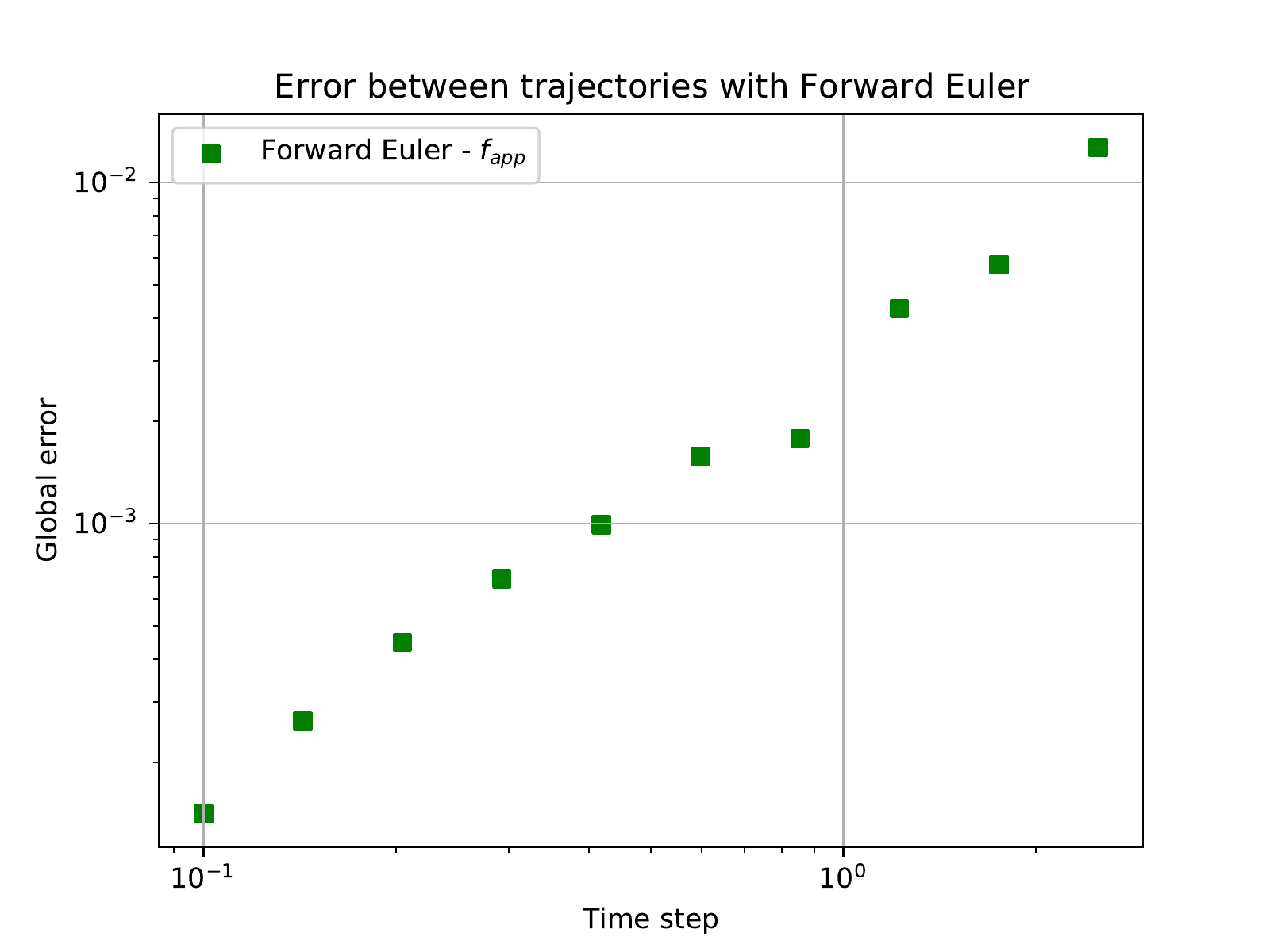}
				\end{figure}
			\end{minipage}
			\begin{minipage}{0.45\linewidth}
				\begin{figure}[H]
					\includegraphics[width=\linewidth]{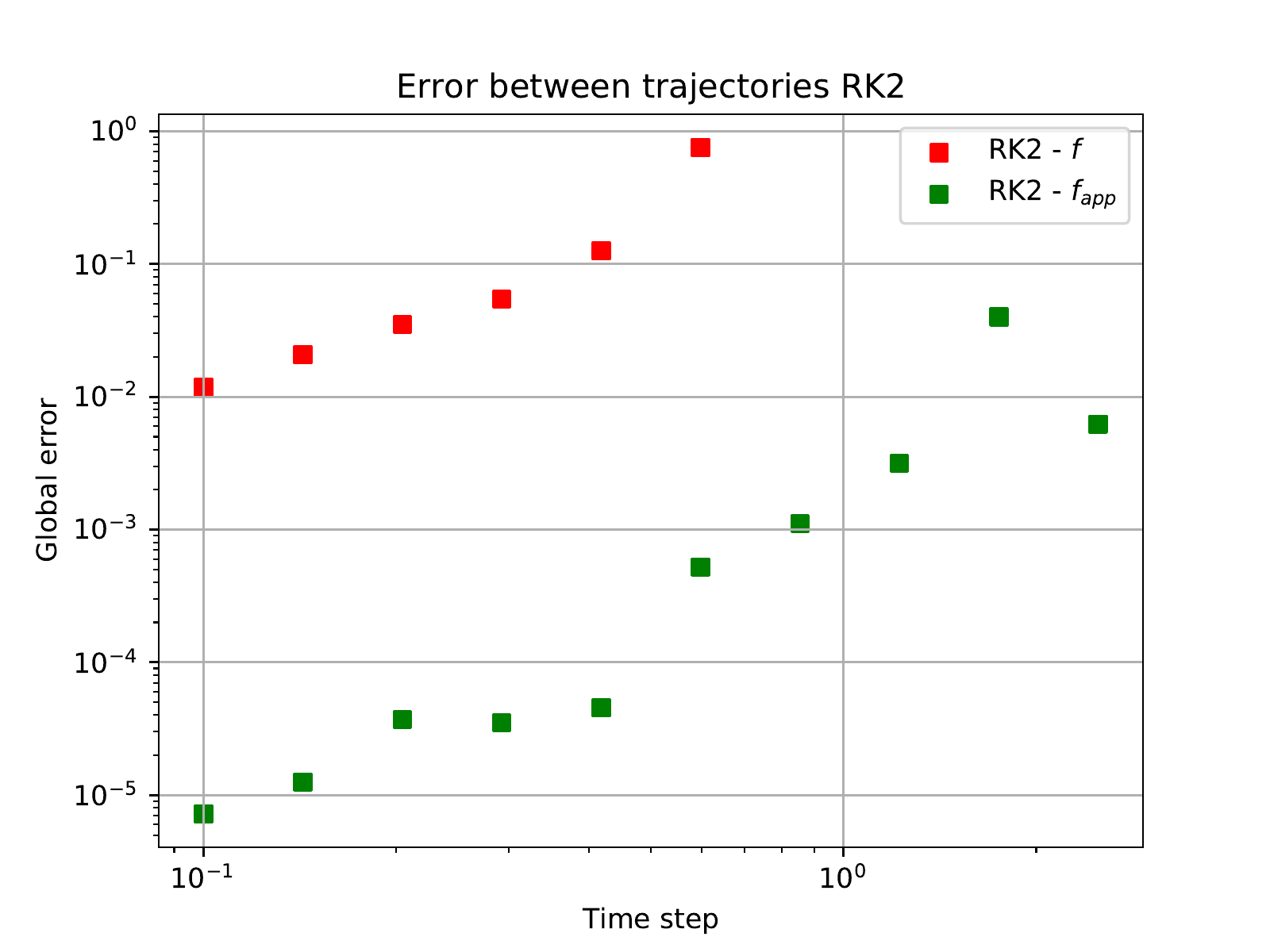}
				\end{figure}
			\end{minipage}
			
		\end{minipage}
	
		\caption{Integration errors (green: integration with $f$, red: integration with $f_{app}(\cdot,h)$). Left: Nonlinear Pendulum with Forward Euler. Right: Nonlinear Pendulum with Runge-Kutta 2. }
		\label{fig_global_error_pendulum}
	\end{figure}

	\begin{figure}[H]
		
		\begin{minipage}{1.1\linewidth}
			\centering
			\begin{minipage}{0.45\linewidth}
				\begin{figure}[H]
					\includegraphics[width=\linewidth]{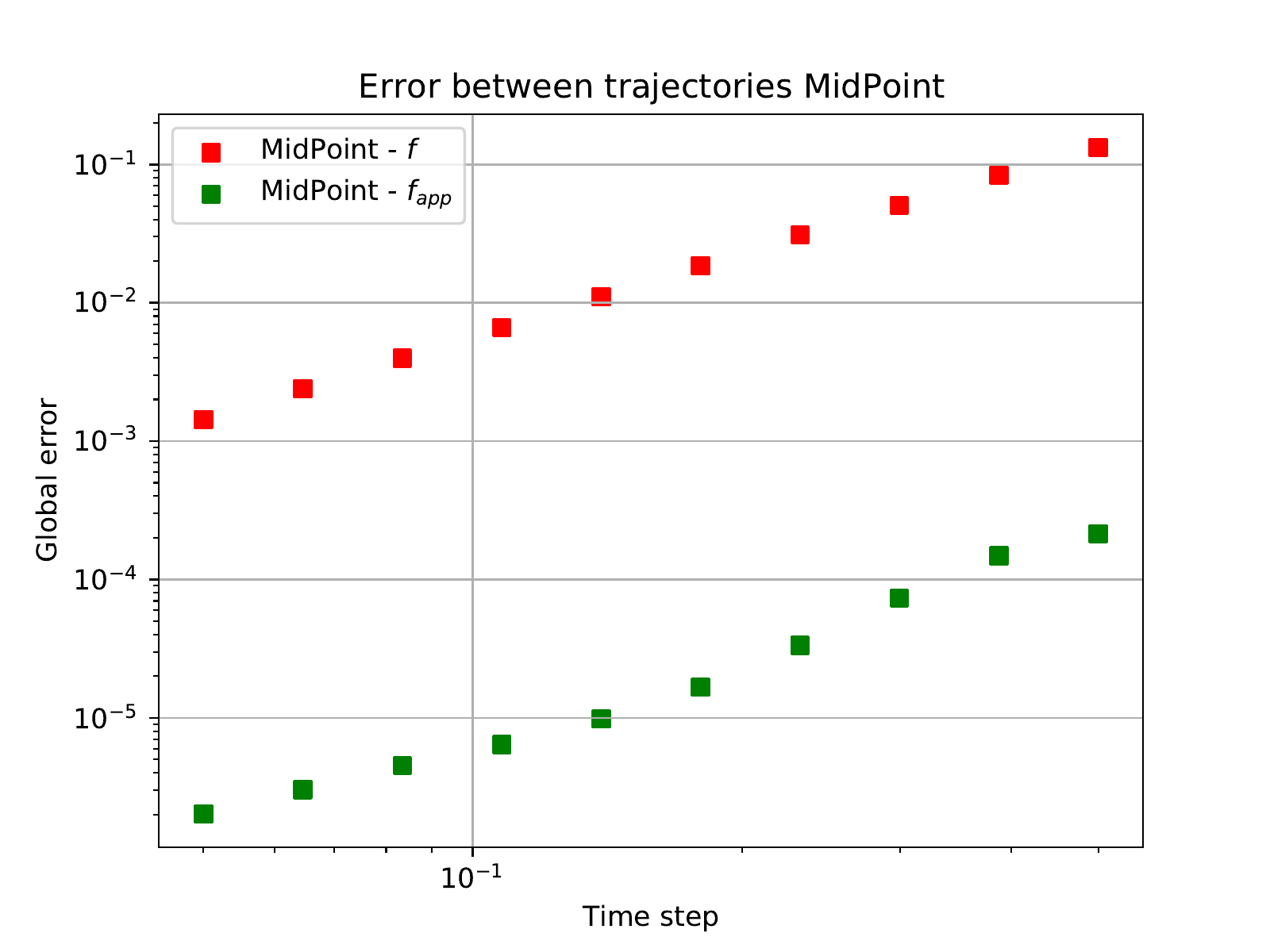}
				\end{figure}
			\end{minipage}
			\begin{minipage}{0.45\linewidth}
				\begin{figure}[H]
					\includegraphics[width=\linewidth]{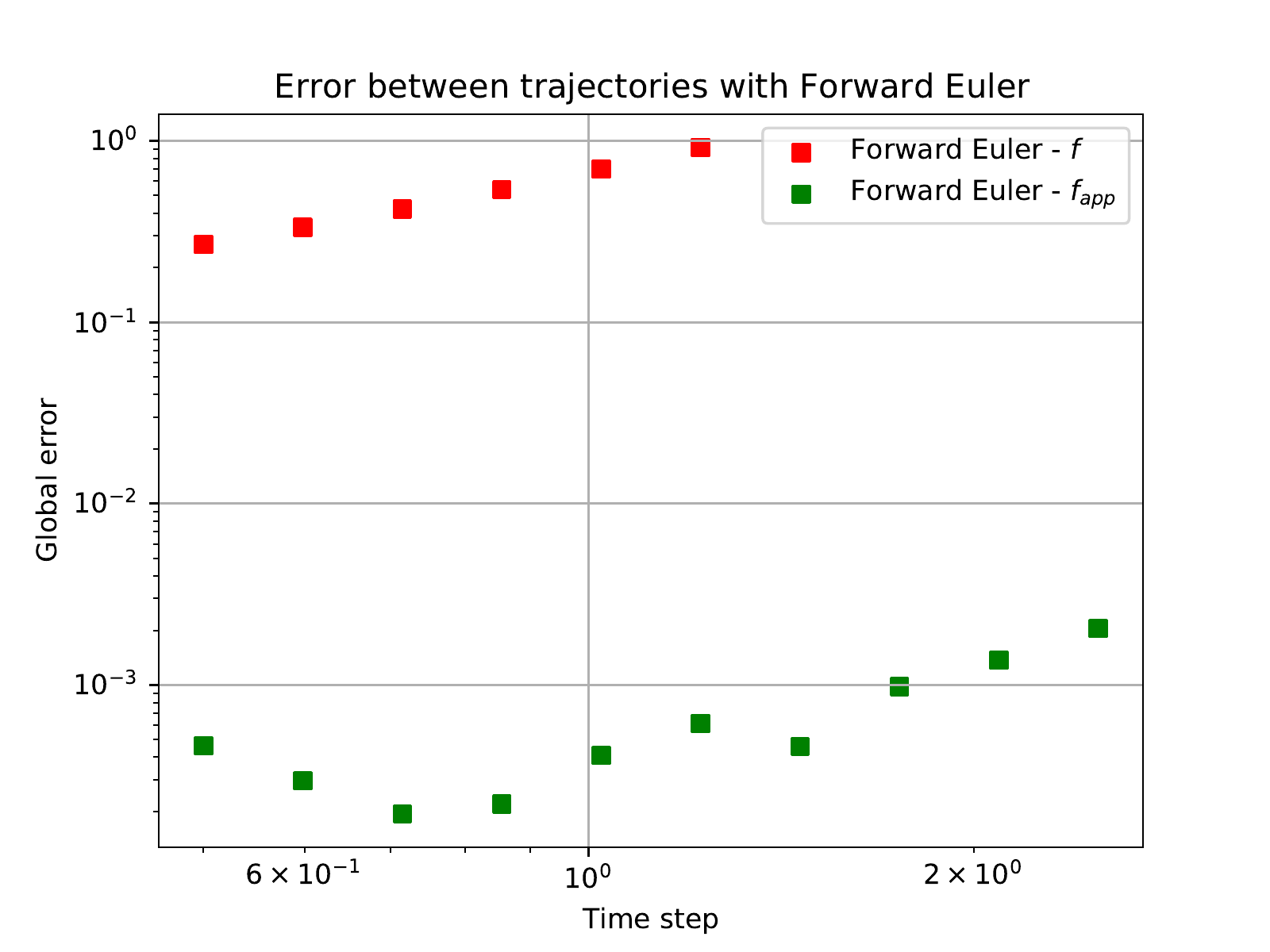}
				\end{figure}
			\end{minipage}
			
		\end{minipage}
		
		\caption{Integration errors (green: integration with $f$, red: integration with $f_{app}(\cdot,h)$). Left: midpoint method for the nonlinear Pendulum. Right: Forward Euler method for the  Rigid Body System. }
		\label{fig_global_error_rigid_body}
	\end{figure}

        Note that the estimate of Theorem \ref{th:gen} is confirmed by Figures \ref{fig_global_error_pendulum} and \ref{fig_global_error_rigid_body} , with a smaller multiplicative constant for the RK2 method.

	\subsection{Computational times for explicit methods}\label{subsection_Time_Computation}
	
	In this subsection, we plot efficiency curves (global error w.r.t. computational time). As the main goal of this paper is to design cheaper and/or more accurate solvers, we shall compare our results with the state-of-the-art  Dormand \& Prince methods \cite{SODE1}.\\

    Figures \ref{fig_time_pendulum} and \ref{fig_time_rigid_body} show numerical errors for explicit methods (Forward Euler and Runge-Kutta 2) with $f_{app}(\cdot,h)$ can be smaller than numerical errors for DOPRI5 with $f$, especially for large time steps.

	\begin{figure}[H]
		
		\begin{minipage}{1.1\linewidth}
			\centering
			\begin{minipage}{0.45\linewidth}
				\begin{figure}[H]
					\includegraphics[width=\linewidth]{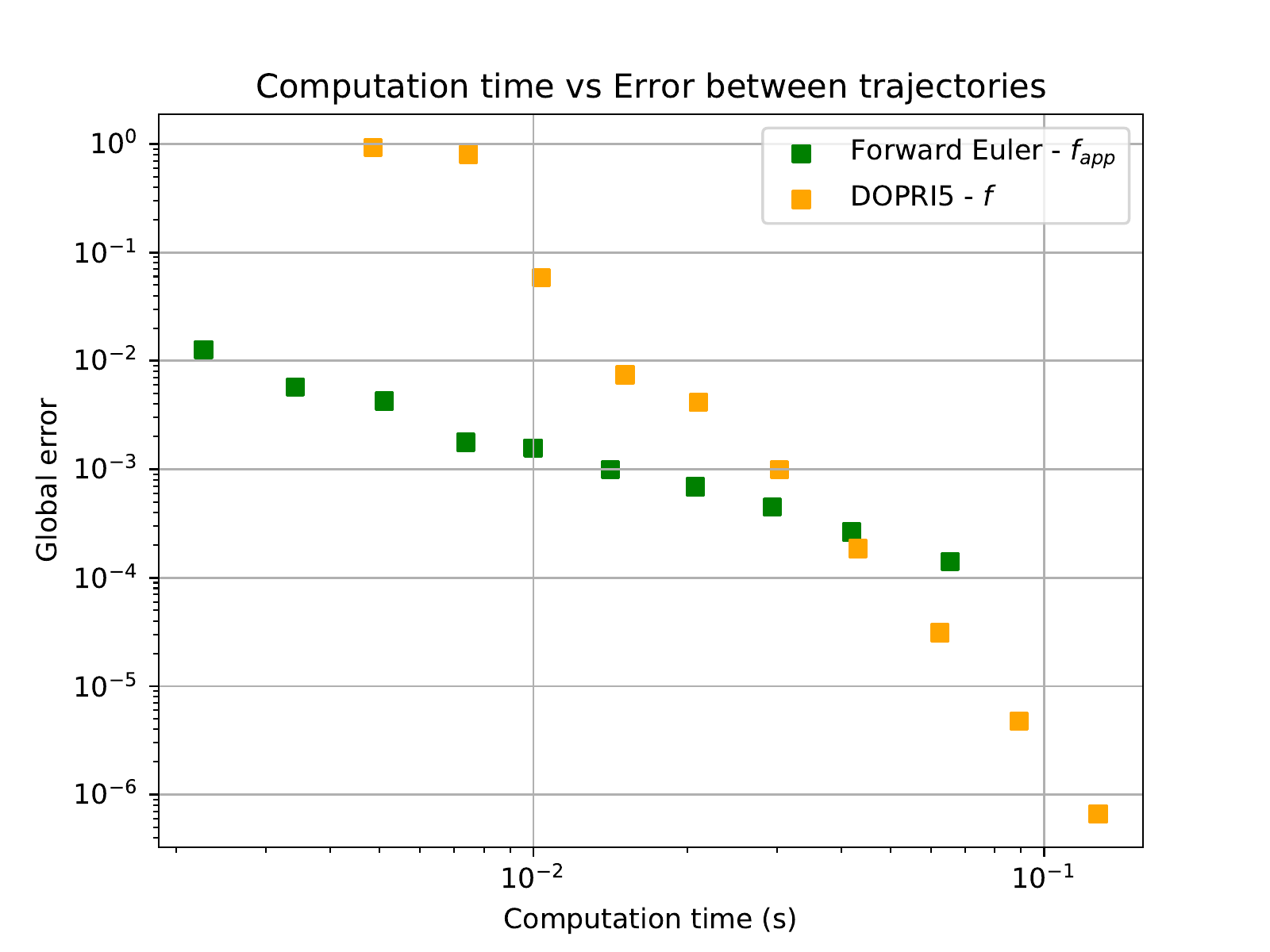}
				\end{figure}
			\end{minipage}
			\begin{minipage}{0.45\linewidth}
				\begin{figure}[H]
					\includegraphics[width=\linewidth]{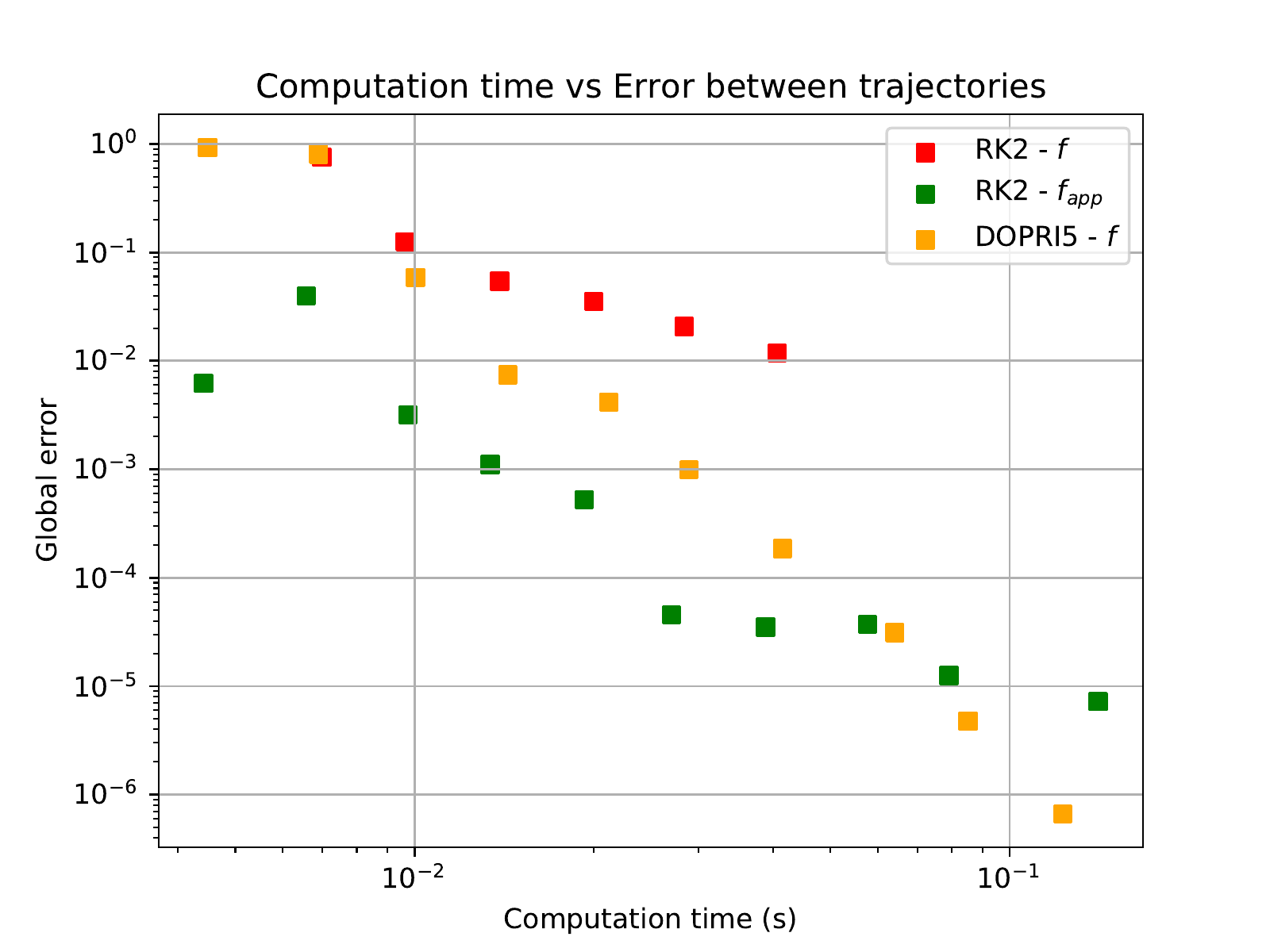}
				\end{figure}
			\end{minipage}
			
		\end{minipage}
		
		\caption{Comparison between computational time and integration error (red: numerical method with $f$, green: integration with $f_{app}(\cdot,h)$, yellow: integration with DOPRI5). Left: Nonlinear Pendulum with Forward Euler. Right: Nonlinear Pendulum with Runge-Kutta 2.}
		\label{fig_time_pendulum}
	\end{figure}

	\begin{figure}[H]
		
		\begin{minipage}{1.1\linewidth}
			\centering
			\begin{minipage}{0.45\linewidth}
				\begin{figure}[H]
					\includegraphics[width=\linewidth]{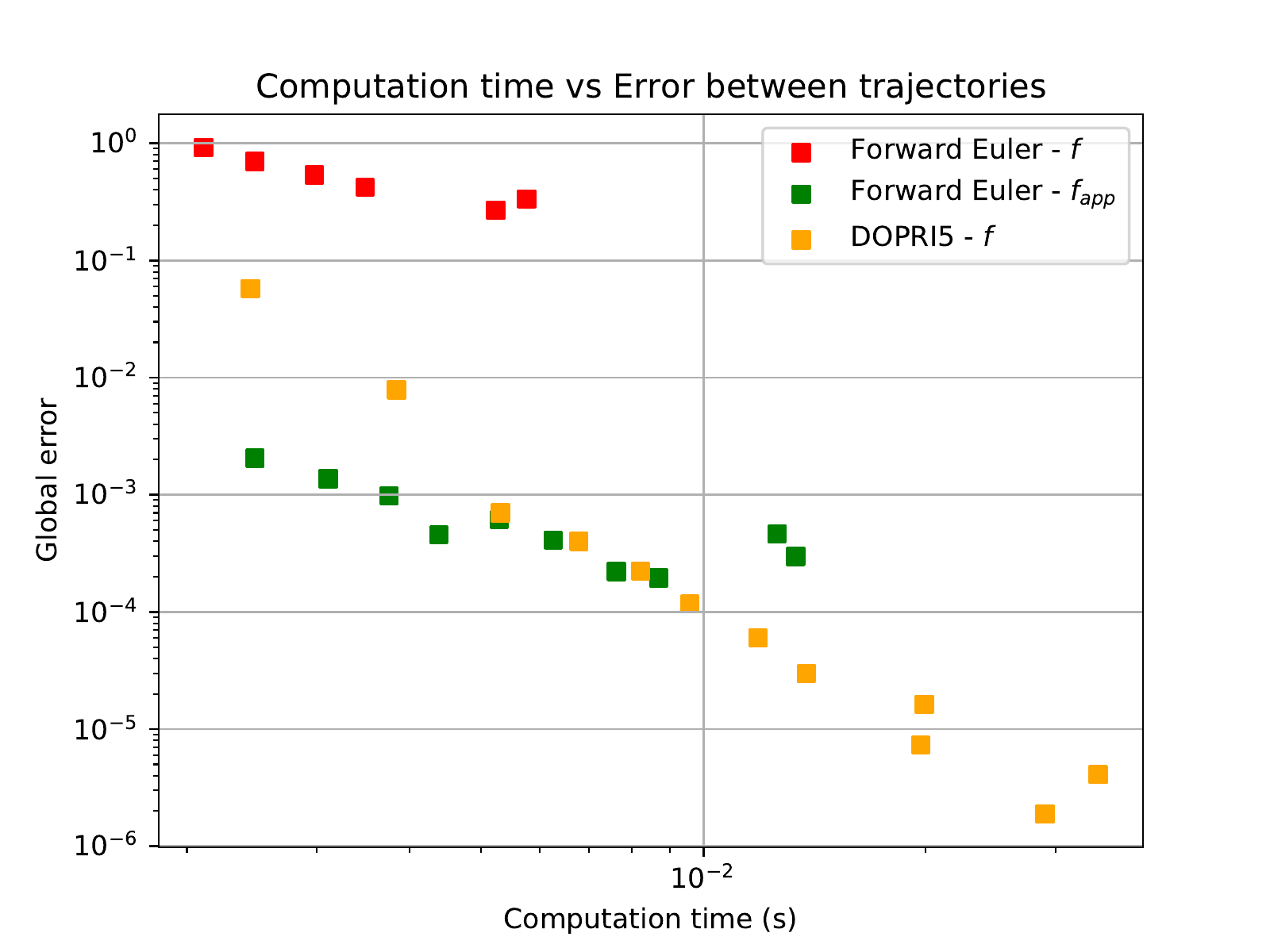}
				\end{figure}
			\end{minipage}
			\begin{minipage}{0.45\linewidth}
				\begin{figure}[H]
				\end{figure}
			\end{minipage}
			
		\end{minipage}
		
		\caption{Comparison between computational time and integration error (red: numerical method with $f$, green: integration with $f_{app}(\cdot,h)$, yellow: integration with DOPRI5) for Rigid Body system with Forward Euler method.}
		\label{fig_time_rigid_body}
	\end{figure}

	Besides, we compare the integration error with the learned modified field and the integration error using truncated modified field at the order $k$ $\tilde{f_h}^k$. For the forward Euler method, use this numerical method with $\tilde{f_h}^k$ will give a numerical method of order $k$, called \textit{modified Euler} \cite{david2021symplectic,GNI}.
	
	\begin{figure}[H]
		
		\begin{minipage}{1.1\linewidth}
			\centering
			\begin{minipage}{0.45\linewidth}
				\begin{figure}[H]
					\includegraphics[width=\linewidth]{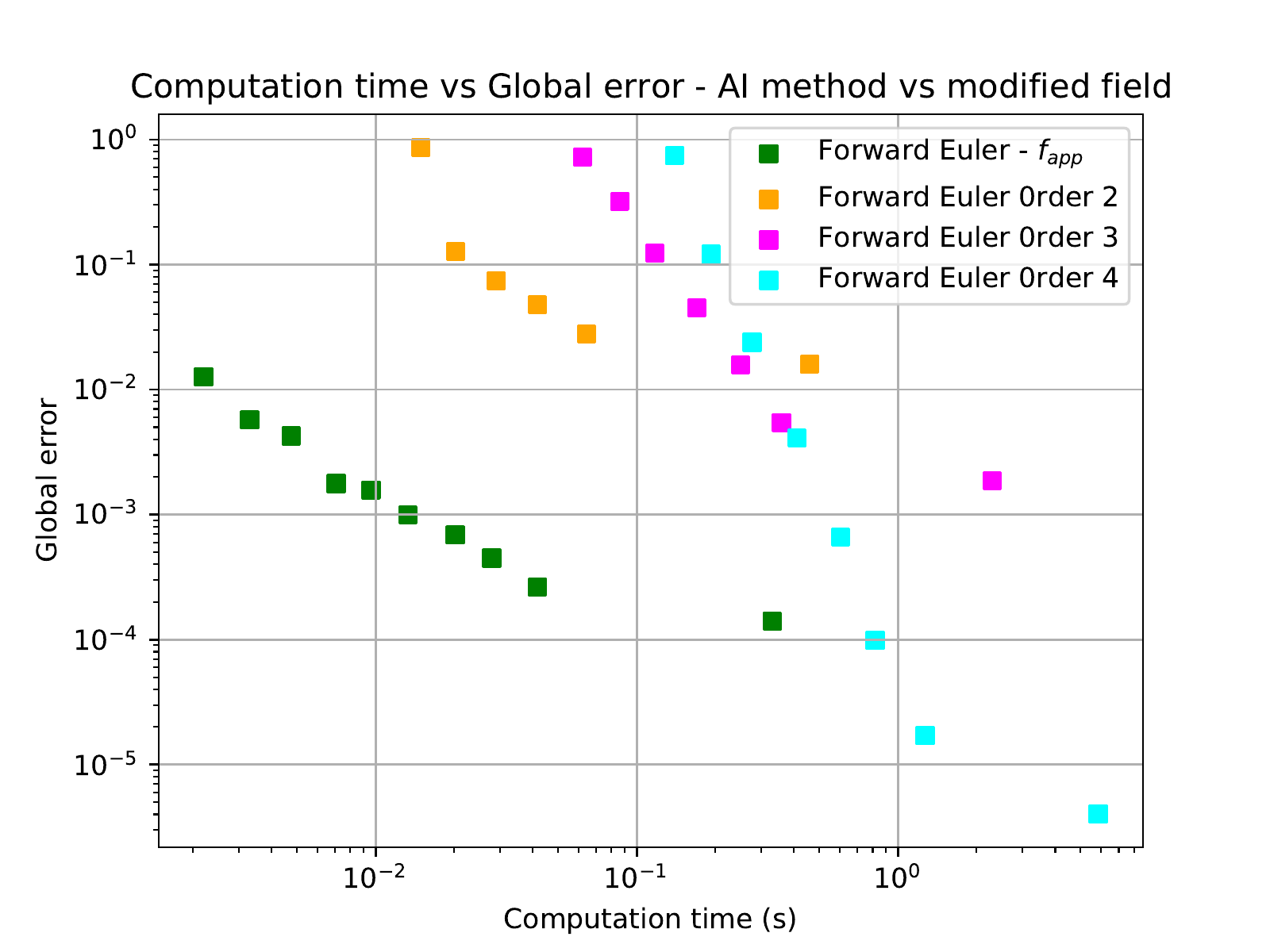}
				\end{figure}
			\end{minipage}
			\begin{minipage}{0.45\linewidth}
				\begin{figure}[H]
					\includegraphics[width=\linewidth]{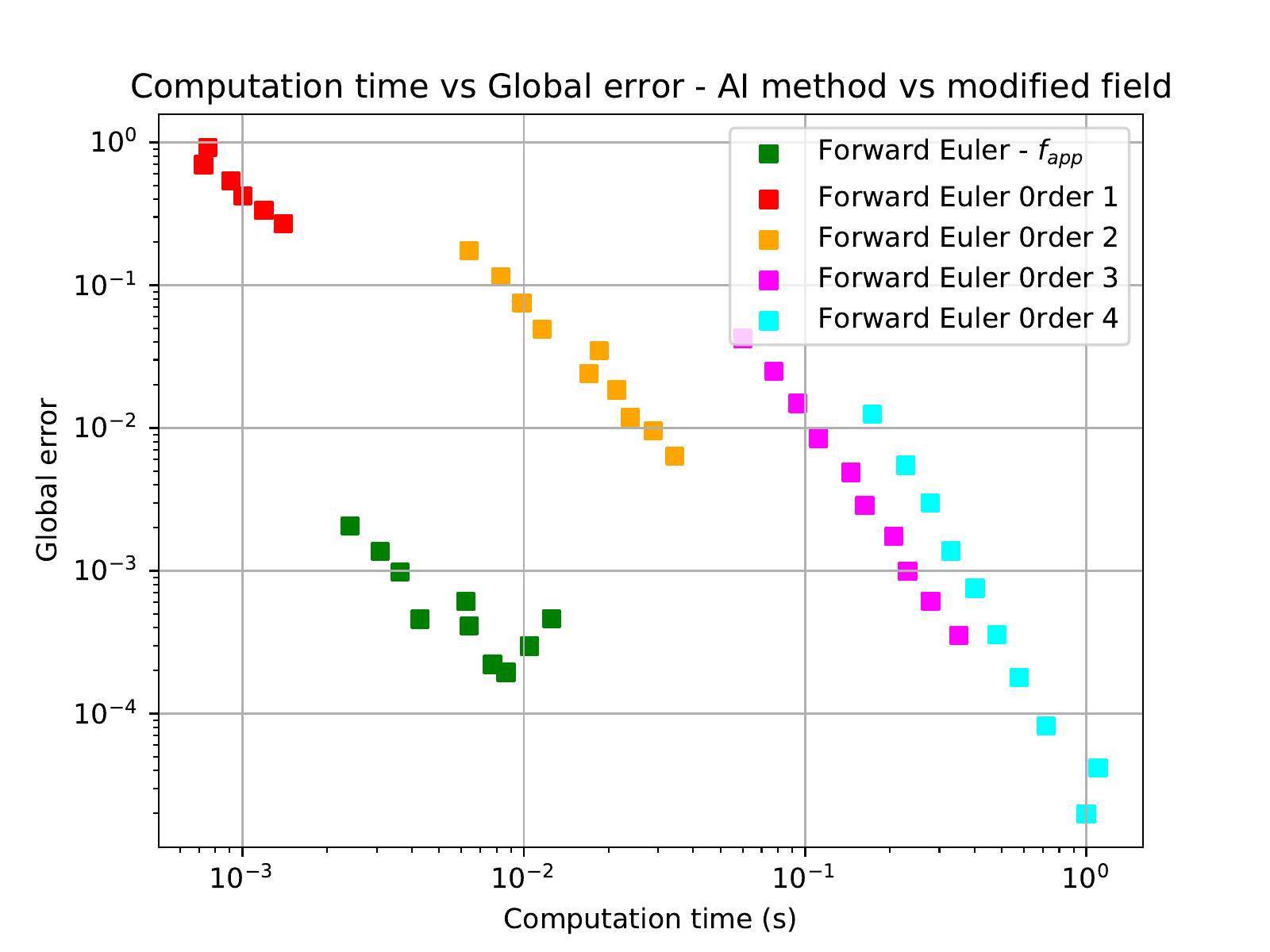}
				\end{figure}
			\end{minipage}
			
		\end{minipage}
		
		\caption{Computational time versus integration error for the Forward Euler method (red: with $f$, yellow: with $\tilde{f_h}^2$, magenta: with $\tilde{f_h}^3$, cyan: with $\tilde{f_h}^4$,  green: with $f_{app}(\cdot,h)$. Left: Nonlinear Pendulum. Right: Rigid Body system. }
		\label{fig_modif}
	\end{figure}
 
        Figure \ref{fig_modif} shows for both systems that integration with $\tilde{f_h}^k$ is more expensive in time or less accurate that integration with $f_{app}(\cdot,h)$. Moreover, we always conserve an interesting integration error whereas the other methods are not accurate at all for small computational time.

        \subsection{Evaluation of alternative method}

        In this subsection, a comparison between the two learning techniques is illustrated. For the comparison to be fair, we adapt the volume  of data so that the training times are nearly identical.\\

        Figure \ref{fig_alternative} shows that both methods are able to generalize for smaller time steps than those used for training. In the two cases, traditional and alternative method with parallel training give approximately identical results

        \begin{figure}[H]
		
		\begin{minipage}{1.1\linewidth}
			\centering
			\begin{minipage}{0.45\linewidth}
				\begin{figure}[H]
					\includegraphics[width=\linewidth]{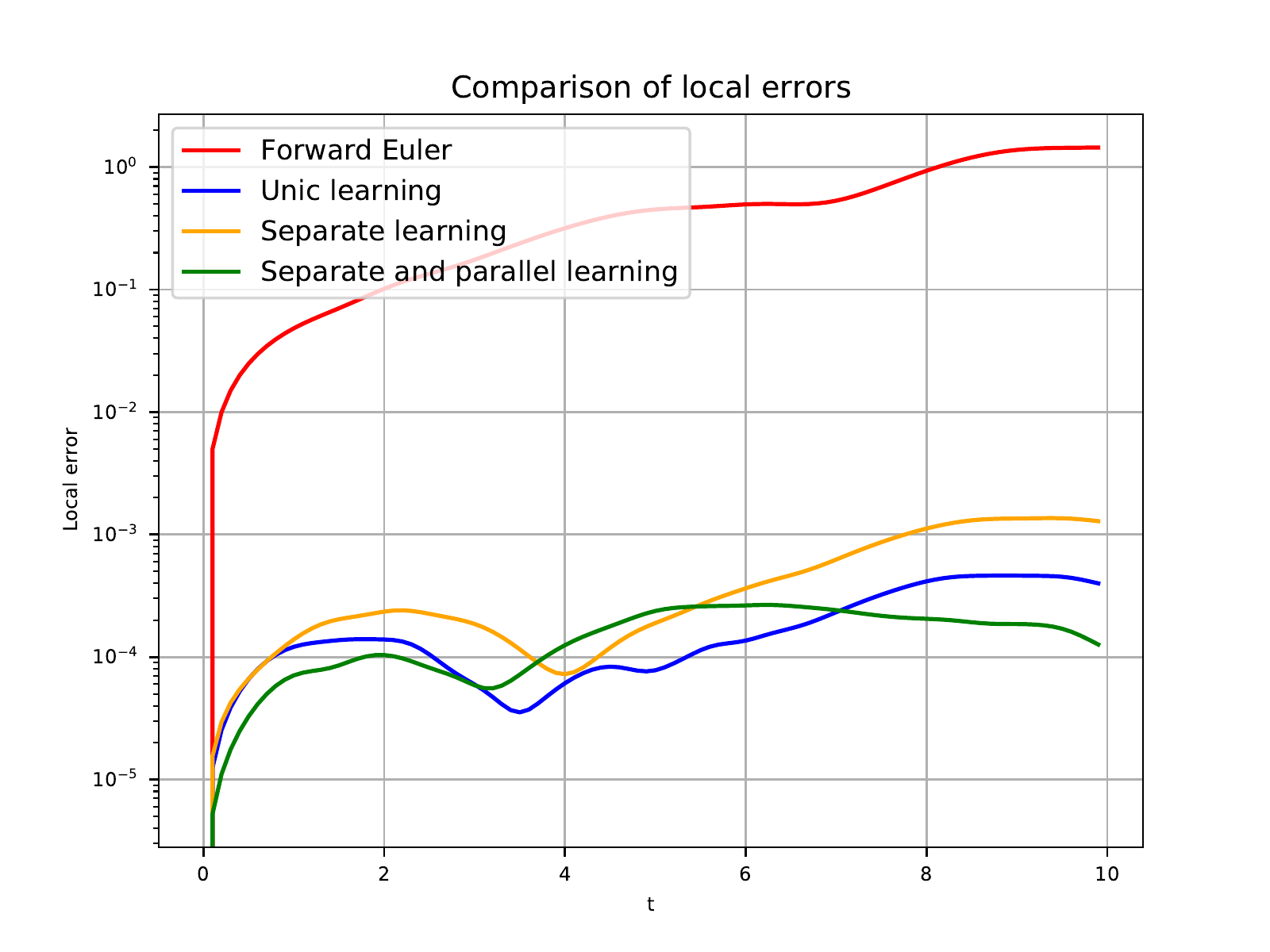}
				\end{figure}
			\end{minipage}
			\begin{minipage}{0.45\linewidth}
				\begin{figure}[H]
					\includegraphics[width=\linewidth]{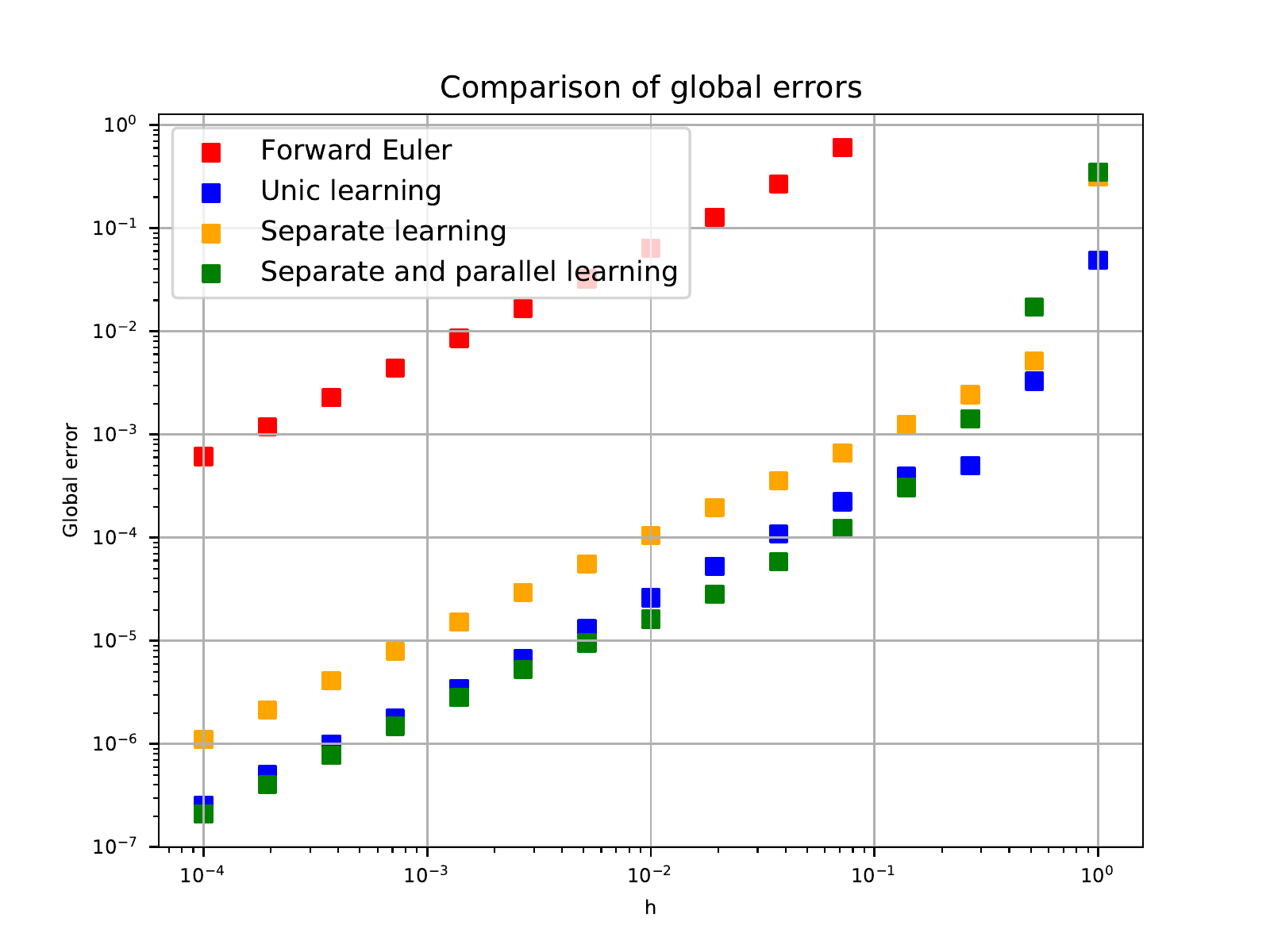}
				\end{figure}
			\end{minipage}
			
		\end{minipage}
		
		\caption{Comparison between forward Euler, traditional and alternative method for simple pendulum. Left: Comparison of local errors. Right: Comparison of integration errors.}
		\label{fig_alternative}
	\end{figure}

	\section{Conclusions}
	
	The numerical experiments presented in this paper demonstrate that learning the modified vector field is very beneficial in terms of efficiency. Clearly, the training of the network is an overload which should be considered separately for real-time applications. Another interesting outcome of our study is the fact that even when explicit formulae of the various terms of the modified field are available, it is advantageous to use its learned counterpart. Eventually, even though the technique used here is not fitted to situations where invariants should be conserved (learning directly the modified Hamiltonian would definitely be a better option), the learned vector field approximately retains the properties of its exact counterpart. It remains to be emphasized that the problems that we solved here are of small dimension and further studies with larges systems (for instance originating from PDEs) are needed. 
	
	\section*{Acknoledgements}
	
	The authors of this paper would like thank Pierre Navaro for his advises in the implementation aspects.

	\appendix
	
	\section{Proof of Theorem \ref{th:gen}}
	
	Let us consider, for time step $h$ and  for all $n \in [\![0,N]\!]$ (where $h = T/N$), the numerical scheme
	\begin{eqnarray*}
    	y_0^*  =  y_0,  \qquad y_{n+1}^* = \Phi_h^{\modapp(\cdot,h)}\left(y_n^*\right) \label{Tokyo_General}
	\end{eqnarray*}
The consistency error is of the form 
\begin{eqnarray*}
    		\eps_n^* & := & y(t_{n+1}) - \Phi_h^{\modapp(\cdot,h)}\left( y(t_n) \right) \label{London_General} \\
    		               & = & \underset{=0 \text{ (Modified field)}}{\underbrace{y(t_{n+1}) - \Phi_h^{\modif}\left( y(t_n) \right)}} + \Phi_h^{\modif}\left( y(t_n) \right)  - \Phi_h^{\modapp(\cdot,h)}\left( y(t_n) \right) \nonumber \\
    		& = & \Phi_h^{\modif}\left( y(t_n) \right)  - \Phi_h^{\modapp(\cdot,h)}\left( y(t_n) \right) \nonumber
		\end{eqnarray*}
		so that, taking $(\ref{NYC_RKE})$ and $(\ref{Toronto_General})$ into account, we have
		\begin{eqnarray*}
    		\left| \eps_n^* \right| & \leqslant & C\delta h^{p+1} .
		\end{eqnarray*}
Upon using $(\ref{Paris_RKE})$, $(\ref{Tokyo_General})$ and $(\ref{London_General})$, the local truncation error can then be written as 
		\begin{eqnarray*}
			e_{n+1}^* & = & \Phi_h^{\modapp(\cdot,h)}(y_n^*) - \Phi_h^{\modapp(\cdot,h)}(y(t_n)) - \eps_n^*,
		\end{eqnarray*}
		and from $(\ref{Beijing_General})$, we get
		\begin{eqnarray*}
			\left| e_{n+1}^* \right|  \leqslant  \left(1+\overline{L} h\right)\left|  e_n^*\right| + \left| \eps_n^* \right| 
			 \leqslant  \left(1+\overline{L} h\right)\left|  e_n^*\right| + C\delta h^{p+1} 
		\end{eqnarray*}
		where $\overline{L}  :=  \underset{h \in [h_-,h_+]}{Max}L_{\modapp(\cdot,h)}$. 
		A discrete Grönwall lemma then leads to 
		\begin{eqnarray*}
			\left|e_n^*\right|  \leqslant   C\delta h^{p+1} \sum_{j=0}^{n-1}e^{\overline{L}(n-j-1)h} 
			 \leqslant  C\delta h^{p+1} \frac{e^{\overline{L} nh}-1}{e^{\overline{L} h}-1} 
			 \leqslant & \frac{C\delta h^{p}}{\overline{L}}\left( e^{\overline{L} T}-1 \right).
		\end{eqnarray*}

	\section{Choice of the parameters}
	
	\subsection{Link between learning error and parameters}\label{Learning_error_parameters}
	
	The influence of the number of parameters over the learning error has been studied. In particular, we have studied the effect of the number of neurons and hidden layers. The dynamical system chosen for the test is the non-linear pendulum while the numerical method is simply the Forward Euler method. In order to approximate the learning error, we compute the value
	\begin{eqnarray}
		\delta & \approx & \underset{h \in H^*}{Max}\frac{1}{h} \lnorm \tilde{f_h}^4(x_{i,j}) - f_{app}(x_{i,j},h²) \rnorm_{l^{\infty}}
	\end{eqnarray}
	where $(x_{i,j})_{0 \leqslant i,j \leqslant 40}$ is a uniform grid on the square $\Omega = [-2,2]^2$, $H^* = (e^{h^*_j})_{0 \leqslant j \leqslant 14}$ where $(h^*_j)_{0 \leqslant j \leqslant 14}$ is a uniform discretization of $[\log(h^-),\log(h^+)]$ and $\tilde{f_h}^4$ corresponds to the field $(\ref{HK_RKE})$ for $k=4$ with $R$ neglected, computed via the formulas given in \cite{chartier2007modified, GNI}.\\
	
	As observed in Figure \ref{fig_learning_error}, a plateau appears when the number of parameters is large. This plateau has a lower value for a larger number of data.  Moreover, we observe that deep networks are more efficient than shallow networks in order to learn the good vector field.
	
	\begin{figure}[H]
		
		\begin{minipage}{1.1\linewidth}
			\centering
			\begin{minipage}{0.45\linewidth}
				\begin{figure}[H]
					\includegraphics[width=\linewidth]{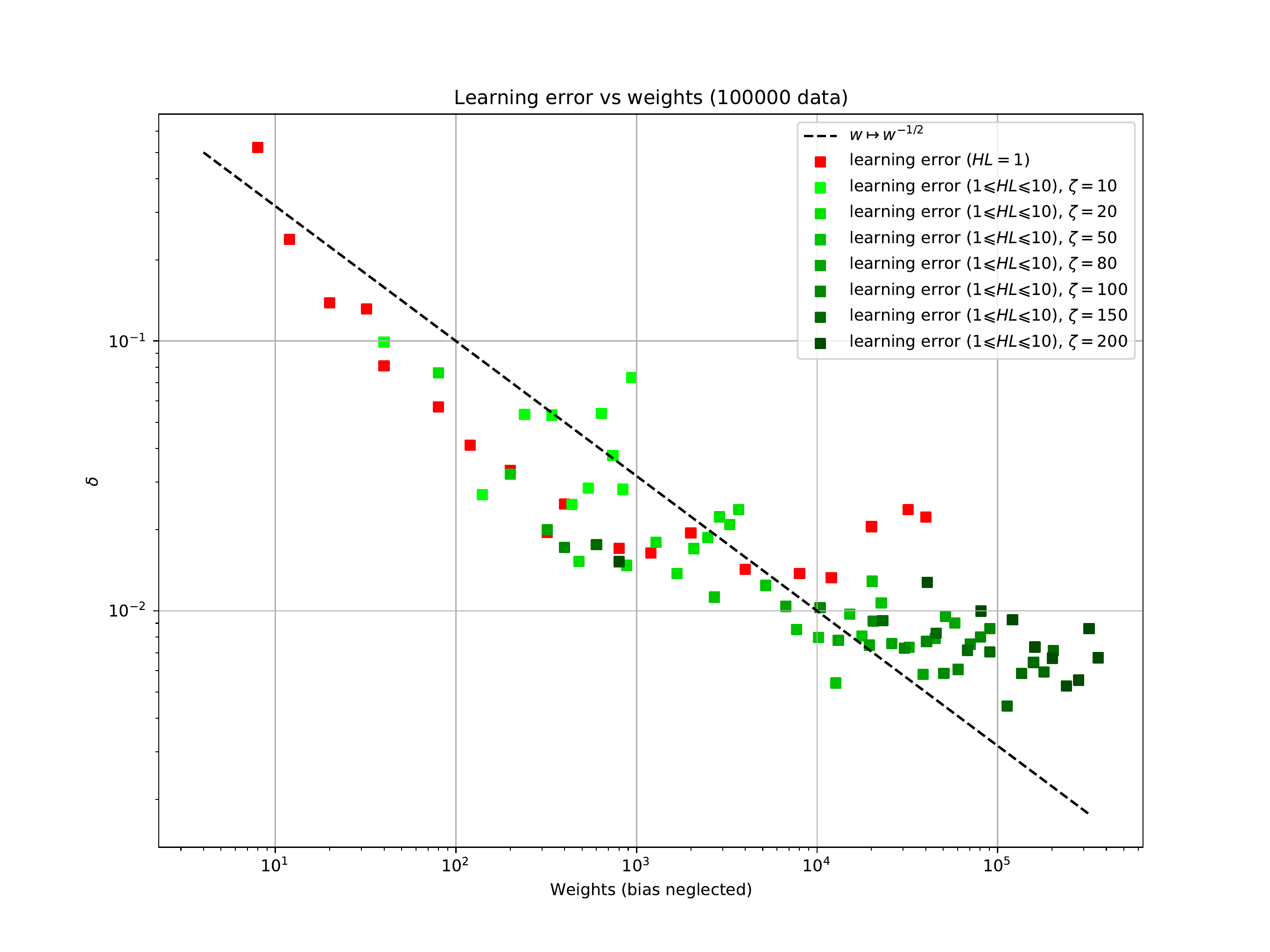}
				\end{figure}
			\end{minipage}
			\begin{minipage}{0.45\linewidth}
				\begin{figure}[H]
					\includegraphics[width=\linewidth]{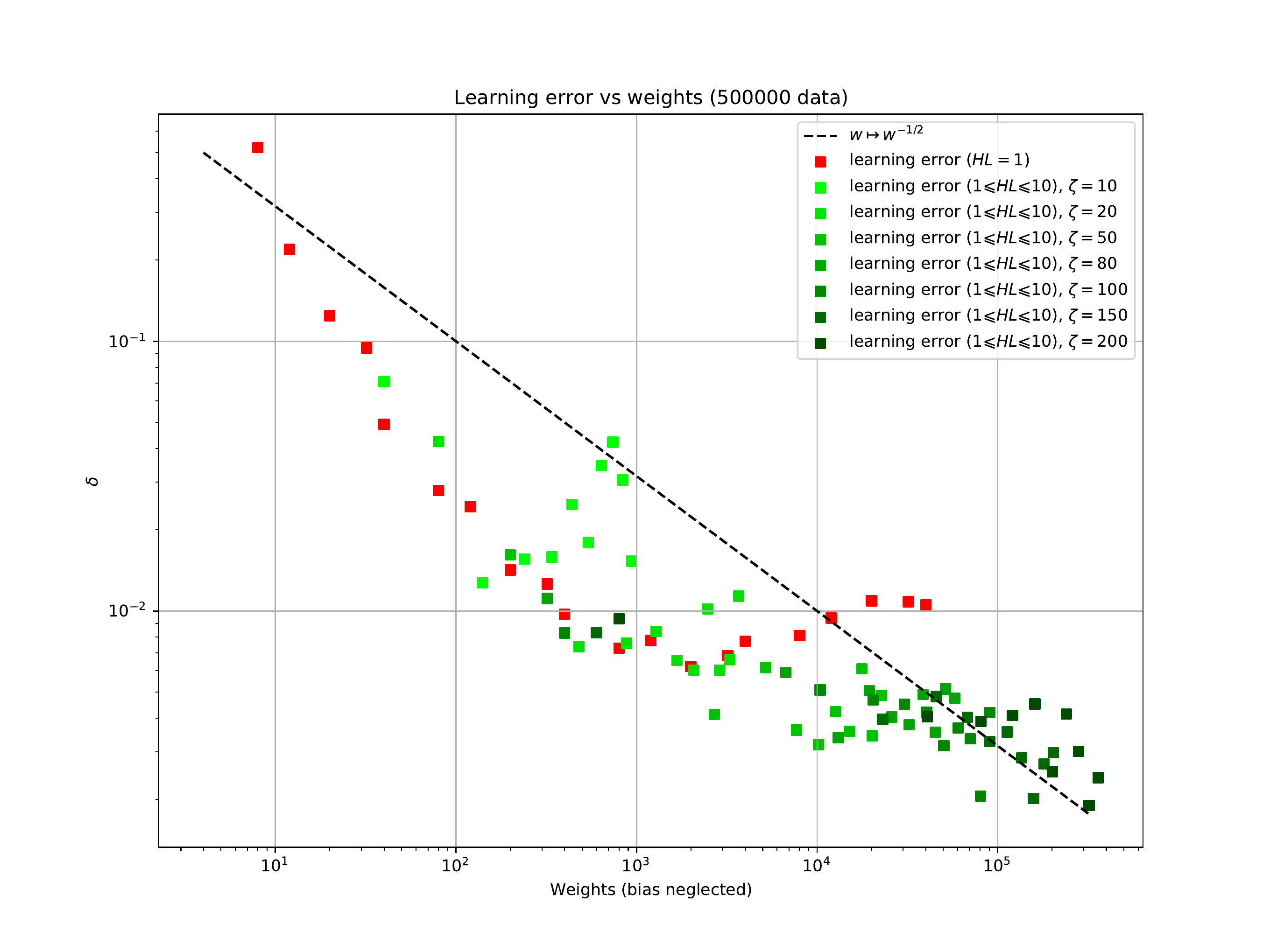}
				\end{figure}
			\end{minipage}
			
		\end{minipage}
		\caption{Learning error $\delta$ versus number of parameters $w$ in the neural network (number of neurons $\zeta$ and hidden layers $HL$). Shallow network is plot with red points whereas deep network is plot with green points (light green for $10$ neurons to dark green for $200$ neurons). Bias are neglected. Learning error is plot for \numprint{100000} data (left) and \numprint{500000} data (right). $200$ epochs are used. $[h^-,h^+] = [10^{-2},10^{-1}]$. The curve of $w \mapsto w^{\frac{1}{2}}$ has been added for comparison purposes.}
		\label{fig_learning_error}
	\end{figure}
	
	\subsection{Parameters selected in our numerical experiments}
	
	For the training, the optimizer \textit{Adam} of Pytorch is used. Besides, the mini-batching option is activated as it appears to be more efficient. Hyperbolic tangent $\tanh$ was selected as activation function.
	
	\subsubsection{Nonlinear Pendulum - Forward Euler}\label{Parameters_Pendulum_Euler}
	
	\begin{center}
	{\tiny 
		\begin{tabular}{| l | l |}
			\hline
			\multicolumn{2}{|c|}{\textcolor{red}{\textbf{Parameters}}} \\
			\hline
			\multicolumn{2}{|l|}{\textcolor{DarkGreen}{\textbf{\# Math Parameters:}}} \\
			\hline
			\textbf{Dynamical system:} & Pendulum \\
			\textbf{Numerical method:} & Forward Euler \\
			\textbf{Interval where time steps are selected:} &  $[h_-,h_+]=[0.1,2.5]$ \\
			\textbf{Time for ODE simulation:} &  $T=20$ \\
			\textbf{Time step for ODE simulation:} & $h=0.1$ \\
			\hline
			\multicolumn{2}{|l|}{\textcolor{DarkGreen}{\textbf{\# AI Parameters:}}} \\
			\hline
			\textbf{Domain where data are selected:} &  $\Omega = [-2,2]^2$ \\
			\textbf{Number of data:} & $K=\numprint{25000000}$ \\
			\textbf{Proportion of data for training:} & $80 \%$ - $K_0 = \numprint{20000000}$ \\
			\textbf{Number of terms in the perturbation (MLP's):} & $N_t = 1$ \\
			\textbf{Hidden layers per MLP:} & $2$ \\
			\textbf{Neurons on each hidden layer:} & $200$ \\
			\textbf{Learning rate:} & $2\cdot 10^{-3}$ \\
			\textbf{Weight decay:} & $1\cdot 10^{-9}$ \\
			\textbf{Batch size (mini-batching for training):} & $300$ \\
			\textbf{Epochs:} & $200$ \\
			\textbf{Epochs between two prints of loss value:} & $20$ \\
			\hline
		\end{tabular}
		}
	\end{center}
	\textbf{Computational time  for training:} $10\text{ h }14\text{ min }17\text{ s}$

	\subsubsection{Rigid body system - Forward Euler}\label{Parameters_Rigid_Body_Euler}
	
	Casimir invariant introduced at the beginning of the section \ref{section_Numerical_experiments} allows to chose the training data in the spherical crown $\left\{ x \in [-2,2]^2 : 0.98 \leqslant |x| \leqslant 1.02 \right\}$
	
	\begin{center}
	{\tiny 
		\begin{tabular}{| l | l |}
			\hline
			\multicolumn{2}{|c|}{\textcolor{red}{\textbf{Parameters}}} \\
			\hline
			\multicolumn{2}{|l|}{\textcolor{DarkGreen}{\textbf{\# Math Parameters:}}} \\
			\hline
			\textbf{Dynamical system:} & Rigid Body \\
			\textbf{Numerical method:} & Forward Euler \\
			\textbf{Interval where time steps are selected:} &  $[h_-,h_+]=[0.5,2.5]$ \\
			\textbf{Time for ODE simulation:} &  $T=20$ \\
			\textbf{Time step for ODE simulation:} & $h=0.5$ \\
			\hline
			\multicolumn{2}{|l|}{\textcolor{DarkGreen}{\textbf{\# AI Parameters:}}} \\
			\hline
			\textbf{Domain where data are selected:} &  $\Omega = \left\{ x \in [-2,2]^2 : 0.98 \leqslant |x| \leqslant 1.02 \right\}$ \\
			\textbf{Number of data:} & $K=\numprint{100000000}$ \\
			\textbf{Proportion of data for training:} & $80 \%$ - $K_0 = \numprint{80000000}$ \\
			\textbf{Number of terms in the perturbation (MLP's):} & $N_t = 1$ \\
			\textbf{Hidden layers per MLP:} & $2$ \\
			\textbf{Neurons on each hidden layer:} & $250$ \\
			\textbf{Learning rate:} & $2\cdot 10^{-3}$ \\
			\textbf{Weight decay:} & $1\cdot 10^{-9}$ \\
			\textbf{Batch size (mini-batching for training):} & $300$ \\
			\textbf{Epochs:} & $200$ \\
			\textbf{Epochs between two prints of loss value:} & $20$ \\
			\hline
		\end{tabular}
		}
	\end{center}
	\textbf{Computational time for training:} $1 \text{ Day }21\text{ h }59\text{ min }51\text{ s}$

	\subsubsection{Nonlinear Pendulum - Runge-Kutta 2}\label{Parameters_Pendulum_RK2}
	
	\begin{center}
	{\tiny 
		\begin{tabular}{| l | l |}
			\hline
			\multicolumn{2}{|c|}{\textcolor{red}{\textbf{Parameters}}} \\
			\hline
			\multicolumn{2}{|l|}{\textcolor{DarkGreen}{\textbf{\# Math Parameters:}}} \\
			\hline
			\textbf{Dynamical system:} & Pendulum \\
			\textbf{Numerical method:} & RK2 \\
			\textbf{Interval where time steps are selected:} &  $[h_-,h_+]=[0.1,2.5]$ \\
			\textbf{Time for ODE simulation:} &  $T=20$ \\
			\textbf{Time step for ODE simulation:} & $h=0.1$ \\
			\hline
			\multicolumn{2}{|l|}{\textcolor{DarkGreen}{\textbf{\# AI Parameters:}}} \\
			\hline
			\textbf{Domain where data are selected:} &  $\Omega = [-2,2]^2$ \\
			\textbf{Number of data:} & $K=\numprint{100000000}$ \\
			\textbf{Proportion of data for training:} & $80 \%$ - $K_0 = \numprint{80000000}$ \\
			\textbf{Number of terms in the perturbation (MLP's):} & $N_t = 1$ \\
			\textbf{Hidden layers per MLP:} & $2$ \\
			\textbf{Neurons on each hidden layer:} & $250$ \\
			\textbf{Learning rate:} & $5\cdot 10^{-4}$ \\
			\textbf{Weight decay:} & $1\cdot 10^{-9}$ \\
			\textbf{Batch size (mini-batching for training):} & $300$ \\
			\textbf{Epochs:} & $200$ \\
			\textbf{Epochs between two prints of loss value:} & $20$ \\
			\hline
		\end{tabular}
		}
	\end{center}
	\textbf{Computational time for training:} $3 \text{ Days }2\text{ h }54\text{ min }35\text{ s}$

	\subsubsection{Nonlinear Pendulum - midpoint}\label{Parameters_Pendulum_midpoint}
	
	\begin{center}
	{\tiny 
		\begin{tabular}{| l | l |}
			\hline
			\multicolumn{2}{|c|}{\textcolor{red}{\textbf{Parameters}}} \\
			\hline
			\multicolumn{2}{|l|}{\textcolor{DarkGreen}{\textbf{\# Math Parameters:}}} \\
			\hline
			\textbf{Dynamical system:} & Pendulum \\
			\textbf{Numerical method:} & midpoint \\
			\textbf{Interval where time steps are selected:} &  $[h_-,h_+]=[0.05,0.5]$ \\
			\textbf{Time for ODE simulation:} &  $T=20$ \\
			\textbf{Time step for ODE simulation:} & $h=0.25$ \\
			\hline
			\multicolumn{2}{|l|}{\textcolor{DarkGreen}{\textbf{\# AI Parameters:}}} \\
			\hline
			\textbf{Domain where data are selected:} &  $\Omega = [-2,2]^2$ \\
			\textbf{Number of data:} & $K=\numprint{20000000}$ \\
			\textbf{Proportion of data for training:} & $80 \%$ - $K_0 = \numprint{16000000}$ \\
			\textbf{Number of terms in the perturbation (MLP's):} & $N_t = 1$ \\
			\textbf{Hidden layers per MLP:} & $2$ \\
			\textbf{Neurons on each hidden layer:} & $200$ \\
			\textbf{Learning rate:} & $2\cdot 10^{-3}$ \\
			\textbf{Weight decay:} & $1\cdot 10^{-9}$ \\
			\textbf{Batch size (mini-batching for training):} & $300$ \\
			\textbf{Epochs:} & $200$ \\
			\textbf{Epochs between two prints of loss value:} & $20$ \\
			\hline
		\end{tabular}
		}
	\end{center}
	\textbf{Computational time for training:} $9\text{ h }47\text{ min }51\text{ s}$

 \subsection{Nonlinear Pendulum - Comparison between traditional and alternative method}

    For the standard method, data are created by simulating several solutions with the same initial condition for different time steps.
     \begin{center}
     {\tiny 
    		\begin{tabular}{| l | l |}
    			\hline
    			\multicolumn{2}{|c|}{\textcolor{red}{\textbf{Parameters}}} \\
    			\hline
    			\multicolumn{2}{|l|}{\textcolor{DarkGreen}{\textbf{\# Math Parameters:}}} \\
    			\hline
    			\textbf{Dynamical system:} & Pendulum \\
    			\textbf{Numerical method:} & Forward Euler \\
    			\textbf{Interval where time steps are selected:} &  $[h_-,h_+]=[0.01,0.5]$ \\
    			\textbf{Time for ODE simulation:} &  $T=20$ \\
    			\textbf{Time step for ODE simulation:} & $h=0.1$ \\
    			\hline
    			\multicolumn{2}{|l|}{\textcolor{DarkGreen}{\textbf{\# AI Parameters:}}} \\
    			\hline
    			\textbf{Domain where data are selected:} &  $\Omega = [-2,2]^2$ \\
                    \textbf{Number of data} & \\
    			\textbf{      - Traditional method:} & $(K,N_h) = (\numprint{50000},5)$ \\
                    \textbf{      - Alternative method:} & $(K,N_h) = (\numprint{76129},5)$ \\
                    \textbf{      - Alternative method (parallel training):} & $(K,N_h) = (\numprint{105735},5)$ \\
                    \textbf{Proportion of data for training:} & $80 \%$ \\
    			\textbf{      - Traditional method:} & $(K_0,N_h) = (\numprint{40000} , 5)$   \\
                    \textbf{      - Alternative method:} & $(K_0,N_h) = (\numprint{60903} , 5)$   \\
                    \textbf{      - Alternative method (parallel training):} & $(K_0,N_h) = (\numprint{84588} , 5)$\\
    			\textbf{Number of terms in the perturbation (MLP's):} & $N_t = 3$ \\
    			\textbf{Hidden layers per MLP:} & $2$ \\
    			\textbf{Neurons on each hidden layer:} & $50$ \\
    			\textbf{Learning rate:} & $2\cdot 10^{-3}$ \\
    			\textbf{Weight decay:} & $1\cdot 10^{-9}$ \\
    			\textbf{Batch size (mini-batching for training):} & $100$ \\
    			\textbf{Epochs:} & $200$ \\
    			\textbf{Epochs between two prints of loss value:} & $20$ \\
    			\hline
    		\end{tabular}
    		}
    \end{center}

    For the alternative method with parallel training, the total computational time for training correspond to the maximum of all training times for each term of the modified field.
    
    \begin{center}
    {\tiny 
         \begin{tabular}{|c|c|c|c|}
            \hline
    	\multicolumn{4}{|c|}{\textbf{Computational time for training}} \\
    	\hline
             Method & Traditional & Alternative & Alternative (parallel training) \\
            \hline
             $f_1$ & & 2 min 54 s & 3 min 24 s \\
             $f_2$ & & 2 min 22 s & 3 min 28 s \\
             $R$ & & 13 min 23 s & 17 min 9 s \\
             \hline
             Total & \textbf{18 min 33 s} & \textbf{18 min 41 s } & \textbf{17 min 9 s} \\
             \hline
         \end{tabular}
         }
    \end{center}
	
\end{document}